\documentclass{article}

    \PassOptionsToPackage{numbers, compress}{natbib}
   \usepackage[preprint]{neurips_2020}

\usepackage[utf8]{inputenc} % allow utf-8 input
\usepackage[T1]{fontenc}    % use 8-bit T1 fonts
\usepackage{hyperref}       % hyperlinks
\usepackage{url}            % simple URL typesetting
\usepackage{booktabs}       % professional-quality tables
\usepackage{amsfonts}       % blackboard math symbols
\usepackage{nicefrac}       % compact symbols for 1/2, etc.
\usepackage{microtype}      % microtypography
\usepackage{graphicx}

\usepackage{amsmath}
\usepackage{amsfonts}
\usepackage{amssymb}
\usepackage{float} 

\usepackage{color}

\usepackage[font=small,labelfont=bf,tableposition=top]{caption}

\DeclareCaptionLabelFormat{andtable}{#1~#2  \& \tablename~\thetable}
\newcommand{\cA}{{\mathcal{A}}}

\newcommand{\cG}{{\mathcal{G}}}

\newcommand{\cD}{{\mathcal{D}}}
\newcommand{\cS}{{\mathcal{S}}}

\newcommand{\cL}{{\mathcal{L}}}
\newcommand{\cX}{{\mathcal{X}}}
\newcommand{\cK}{{\mathcal{K}}}

\newcommand{\cT}{{\mathcal{T}}}

\newcommand{\cO}{{\mathcal{O}}}
\newcommand{\cH}{{\mathcal{H}}}

\newcommand{\cY}{{\mathcal{Y}}}

\newcommand{\be}{{\textbf{e}}}

\newcommand{\bX}{\textbf{X}}
\newcommand{\bY}{\textbf{Y}}

\newcommand{\bd}{\textbf{d}}
\newcommand{\bq}{\textbf{q}}

\newcommand{\bD}{\textbf{D}}

\newcommand{\bA}{\textbf{A}}

\newcommand{\bz}{\textbf{z}}

\newcommand{\br}{\textbf{r}}

\newcommand{\bw}{\textbf{w}}
\newcommand{\bv}{\textbf{v}}
\newcommand{\bap}{\pmb{\alpha}} 
\newcommand{\bpii}{\pmb{\pi}} 
\newcommand{\bPi}{\pmb{\Pi}} 

\newcommand{\bld}{{\pmb{\lambda}}}
\newcommand{\beps}{{\pmb{\epsilon}}}

\newcommand{\bpsi}{{\pmb{\psi}}}
\newcommand{\bPsi}{{\pmb{\Psi}}}
\newcommand{\bphi}{{\pmb{\phi}}}
\newcommand{\bmu}{{\pmb{\mu}}}

\newcommand{\bbR}{\mathbb{R}}
\newcommand{\bbE}{\mathbb{E}}
\newcommand{\bbP}{\mathbb{P}}

\newcommand{\bbI}{\mathbb{I}}

\newenvironment{proof}[1][Proof]{\noindent\textbf{#1.} }{\ \rule{0.5em}{0.5em}}

\usepackage{mathtools}

\DeclarePairedDelimiter\floor{\lfloor}{\rfloor}

\def\htien#1{}

\newcommand{\T}{\text{\tiny T}}

\newcommand{\RE}{\textsc{RE}}
\newcommand{\SR}{\textsc{S}}
\newcommand{\CC}{\textsc{CT}}
\newcommand{\CT}{\textsc{CT}}
\newcommand{\DS}{\textsc{DS}}
\newcommand{\EV}{\textsc{EV}}

\newcommand{\BD}{\text{BD}}
\newcommand{\KL}{\text{KL}}

\newtheorem{theorem}{Theorem}%[section]

\newtheorem{definition}{Definition}

\newtheorem{proposition}{Proposition}%[section]

\title{A Relation Analysis of Markov Decision Process Frameworks}

\author{%
  Tien Mai\thanks{Corresponding author} \\
 SIS, Singapore Management University\\
  \texttt{atmai@smu.edu.sg} \\
 \And
  Patrick Jaillet \\
  EECS, Massachusetts Institute of Technology\\
  \texttt{Jaillet@mit.edu} \\
}

\begin{document}

\maketitle

\begin{abstract}
 We study the relation between different Markov Decision Process (MDP) frameworks in the machine learning and econometrics literatures, including the standard MDP, the entropy and general regularized  MDP, and stochastic MDP, where the latter is based on the assumption that the reward function is stochastic and follows a given distribution. We show that the entropy-regularized MDP is equivalent to a stochastic MDP model, and  is strictly subsumed by the general regularized MDP. Moreover, we propose a distributional stochastic MDP framework by assuming that the distribution of the reward function is ambiguous. We further show that the distributional stochastic MDP is equivalent to the regularized MDP, in the sense that they always yield the same optimal policies. We also provide a connection between stochastic/regularized MDP and constrained MDP. Our work gives a unified view on several important MDP frameworks, which would lead new ways to interpret the (entropy/general) regularized MDP frameworks through the lens of stochastic rewards and vice-versa. Given the recent popularity of  regularized MDP in (deep) reinforcement learning, our work brings new understandings of how such algorithmic schemes work and suggest ideas to develop new ones. 
\end{abstract}

\section{Introduction}

Markov decision processes (MDPs) are appealing for various problems involving sequential decision-making. Beside the standard MDP framework \cite{Puterman2014markov}, which is already popular  in the machine learning, operations research and econometrics communities, the literature  has seen several variant MDP frameworks proposed over the recent decades, with various motivations. For example, relative entropy regularizers have been added to the reward function  to improve exploration and compositionality in reinforcement learning (RL), e.g., Soft-Q learning \citep{Fox2015taming,Schulman2017equivalence,Haarnoja2017RL_ER} and Sotf-Actor-Critic \citep{Haarnoja2018softa,Haarnoja2018softb}. Kullback-Leibler (KL) divergence regularizers have been used  to penalize the ``\textit{divergence}'' between two conservative policies, with the motivation of preventing earlier convergence to sub-optimal policies in a policy iteration algorithm. This idea  have been implemented in many state-of-the-art RL algorithms, e.g., Dynamic Policy Programming (DPP) \citep{Azar2012dynamic},  Maximum A Posteriori
Policy Optimization (MPO) \cite{Abdolmaleki2018maximum,Abdolmaleki2018relative,Mankowitz2019robust}.
\cite{Lee2018sparse,Chow2018path} add  a Tsallis entropy to the Markov problem with the motivation of having sparse and non-deterministic policies, and a theory a general regularized MDP framework has been studied in \cite{Geist2019theory}.      
The use of  entropy-regularizers in MDP is also advantageous in removing ambiguity in some imitation learning/inverse reinforcement learning (IRL) applications \citep{ziebart2010_IRL_Causal,Ho2016generative,Fu2017RobustIRLlearning,Finn2016_connectionIRL}.  

In the econometrics community, the stochastic  MDP framework (a.k.a dynamic discrete choice modeling) have been proposed and intensively studied since 1980s  \citep{Rust1987GMC,Rust1988MLE} (see \cite{Aguirregabiria2010dynamic} for a survey). This framework is based on the assumption that the reward function is no-longer deterministic but stochastic and follows a given distribution, and decisions are made according to each realization of the random rewards. The main motivation is to overcome the deterministic notation of optimal policies resulting from the standard MDP by allowing  the agent to make different decisions when the rewards vary, thus improve exploration. 
Such stochastic Markov models, under some settings, yield randomized/soft policies, which is helpful in casting some behavioral learning procedures as probabilistic inference problems \citep{Rust1988MLE}. From this viewpoint, the stochastic MDP framework shares some common motivations with the entropy-regularized MDP one \cite{ziebart2010_IRL_Causal}.    

In this paper, we make the first attempt to thoroughly study the relation between aforementioned MDP frameworks, aiming at providing new ways to view each MDP modeling scheme through different angles and perspectives. We consider a broad class of MDP frameworks, including the standard MDP, entropy-regularized MDP \citep{ziebart2010_IRL_Causal}, general regularized MDP \citep{Geist2019theory}, stochastic MDP with i.i.d Gumbel distributed rewards (a.k.a logit dynamic discrete choice), and general stochastic MDP \citep{Rust1987GMC}. We summarize our contributions as follows.
\begin{itemize}
    \item [(i)] 
We  propose a new MDP framework, called \textit{distributional stochastic MDP}, by assuming that the distribution of the reward function in stochastic MDPs is not given certainly, but belongs to an ambiguity set of distributions.  We then show that the corresponding Markov problem is more general and more tractable, as compared to the standard stochastic MDP, 
in the sense an optimal policy can be found by value iteration and convex optimization, under some distributional settings with marginal information.
\item[(ii)] We show that, interestingly, there is a nested relation between these MDP frameworks, that is, (i) the entropy-regularized MDP and  
stochastic MDP with i.i.d Gumbel distributed rewards are ``\textit{equivalent}'' and strictly covered by the stochastic MDP, and (ii) the general regularized MDP and the distributional stochastic MDP are equivalent and strictly cover other frameworks. Here,  ``\textit{equivalence}'' 
%means that the two frameworks contain MDP models that always yield common optimal policies for any reward function, and also 
means that the uses of these MDP frameworks are equivalent, in both RL and IRL perspectives. There are many interesting points that can be seen from this nested relationship, namely, we can view (entropy) regularized MDP through the lens of stochastic-reward MDP and vice-versa. This also indicates  the robustness of regularized MDP from a viewpoint  of stochastic-reward with ambiguous distribution. The results also lead to new regularizers  with interpretations from stochastic-reward perspectives. 
\item[(iii)] We also make a connection between regularized/stochastic MDP and constrained MDP \cite{Altman1999constrained} by showing that any aforementioned MDP model can be \textit{equivalently} represented by a constrained MDP model. 
%We believe that this finding gives an unified way to view/interpret other frameworks. 
That is, adding regularizers or random noises to the reward function is \textit{equivalent} to imposing some constraints on the desired optimal policies. We also provide discussions on how this constrained-MDP viewpoint would be beneficial to develop generalized RL algorithms.%% or control the sparsity of MDPs.         
\end{itemize}
The equivalence between entropy-regularized MDP and logit dynamic discrete choice in the finite-horizon case was recognized previously by \cite{Ermon2015learning} and in our work we show it holds for the infinite case as well. Indeed, we cover a much broader class of MDP frameworks. Our work generally gives a unified picture of several useful MDP frameworks in the literature, thus brings new understandings of how they work and  suggests ideas to develop new algorithms.     

The paper is structures as follows. In Section \ref{sec:problem-description} we revisit some well-known MDP frameworks in the literature. 
Section \ref{sec:DS-MDP} introduces the distributional stochastic MDP framework.
Section \ref{sec:nested-relation} presents the nested relationship between regularized and stochastic MDP. We show a connection between constrained MDP and other frameworks in Section \ref{sec:CT-MDP} and finally, Section \ref{sec:conclude} concludes. Proofs and some relevant discussions are given in the supplementary material. We use $|\cS|$ to denote the cardinality of set $\cS$. Boldface characters represent matrices (or vectors) or a collection of values/functions. 

\section{Preliminaries}\label{sec:problem-description}
Consider a infinite-horizon Markov decision process (MDP) for an agent with finite states and actions, defined as $(\cS,\cA,\bPi,\bq, \br,\gamma)$, where $\cS$ is a set of states $\cS = \{1,2,\ldots,|\cS|\}$, $\cA$ is a finite set of actions,
$\bPi = \{\bpii^0,\bpii^1,\ldots\}$ is a policy function 
where $\pi^t(a_t|s_t)$ is the probability of making action $a_t\in\cA$ in state $s_t\in\cA$ at time $t\in \{0,\ldots,\infty\}$, 
$\bq:\cS\times \cA\times\cS \rightarrow [0,1]$ is a transition probability function, i.e., $q(s'_t|a_t,s_t)$ is the probability of moving to state $s'_t\in\cS$ from $s_t\in \cS$ by performing action  $a_t\in \cA$ at time $t$, $\br$ is a vector of reward functions such that $r(a_t|s_t)$ is a reward function 
associated with state $s_t\in\cS$ and action $a_t\in\cA$ at time $t$, and $\gamma\in[0,1]$ is a discount factor. 

In the following we review different MDP frameworks, including the standard one,  regularized MDP \citep{ziebart2010_IRL_Causal}, stochastic MDP \citep{Rust1988MLE,Rust1987GMC} (a.k.a dynamic discrete choice). We also briefly discuss how optimal policies can be computed. They are important ingredients to study the relation. % between these MDP formulations.

\textbf{{Standard MDPs.}}
In a standard MDP model, the aim  is to find an optimal sequential policy $\bPi$ that maximizes the expected discounted reward
\begin{align}
\sup_{\substack{\bpii^t \in \Delta^\pi \\ t=0,\ldots,\infty}}\Bigg\{ \bbE_{\tau \sim (\bPi,\bq)}\Bigg[\sum_{t=0}^{\infty}\gamma^{[t]} r(a_t|s_t)  \Bigg]\Bigg\},\label{prob:DET-Max-EP}
\end{align}
where $\tau$ is a trajectory $\tau = (s_0,a_0,\ldots,s_{\infty},a_{\infty})$ and  $\Delta^\pi$ is the set of all possible policy functions, i.e., $\Delta^\pi = \{\bpii|\ \pi(a_t|s_t)\geq 0,\ \sum_{a_t}\pi(a_t|s_t) = 1\}$. 
Here, $\gamma^{[t]}$ refers to \textit{``$\gamma$ to the power of $t$''} to distinguish it from a superscript $t$. 
%he Markov problem can be solved by value iteration or policy iteration, and 
It is well-known the problem  admits a deterministic notion of optimality, i.e., optimal policies are deterministic. 

\textbf{{Regularized MDPs.} }
In a general regularized MDP model \citep{Geist2019theory}, a set of regularizers  $\bphi$  are added to the rewards; $\bphi = \{\phi_{s}(\bpii_{s}),\ s\in\cS \}$, where $\bpii_{s} = \{\pi(a|s),\forall a\in \cA\}$, for any $s\in \cS$.
The Markov decision problem can be stated as
\begin{align}
\sup_{\substack{\bpii^t \in \Delta^\pi \\ t=0,\ldots}}\Bigg\{ \bbE_{\tau \sim (\bPi,\bq)}\Bigg[\sum_{t=0}^{\infty}\gamma^{[t]} \Big( r(a_t|s_t) + \phi_{s_t}(\bpii^t_{s_t}) \Big)\Bigg]\Bigg\}.\label{prob:Regularized-MDPs}
\end{align}
It is typically assumed that $\phi_s(\bpii_s)$ is (strictly) concave. 
%As mentioned, regularizers are added to the Markov model for various motivations such as improving exploration with soft/randomized policies in RL \cite{Fox2015taming,Haarnoja2017RL_ER,Haarnoja2018softa,Haarnoja2018softb}, or preventing earlier convergence to a sub-optimal policies in a policy iteration algorithm \citep{Schulman2015trust,Geist2019theory}. 
To compute an optimal policy, we define, for any $s\in \cS$, the value functions
\begin{align}
    V^{\RE,\bphi}(s) = \sup_{\substack{\bpii^t \in \Delta^\pi\\t=0,\ldots}} \bbE_{\tau \sim (\bPi,\bq)}\Bigg[\sum_{t=0}^{\infty}\gamma^{[t]} \Big( r(a_t|s_t) + \phi_{s_t}(\bpii^t_{s_t}) \Big)\Bigg| s_0 = s\Bigg];%\ V^{\RE,\bphi}(s) =  V^{\RE,\bphi}_{\bpii}(s),\nonumber
\end{align}
and the Bellman update
$
\cT^{\RE,\bphi}[V](s) = \sup_{\bpii_s}\bbE_{\bpii_s}[ r(a|s)+ \gamma\sum_{s'}q(s'|s,a)V(s')]  + \phi_s(\bpii_s).%\nonumber    
%\cT^{\RE,\bphi}[V](s) &=  %\cT^{\RE,\bphi}_{\bpii}[V](s).\nonumber
$
Then one can show that  $\cT^{\RE,\bphi}[V]$ is a contraction \citep{Geist2019theory},  guaranteeing that the Bellman equation always have a unique fixed point solution. Furthermore, 
 $V^{\RE,\bphi}$ defined above is the unique  solution to  $\cT^{\RE,\bphi}[V] = V$. Moreover, if we denote by $\phi^*_s(\bw_s)$ the Legendre-Fenchel transform (or convex conjugate) of $-\phi_s(\bpii_s)$, i.e.,
\begin{equation}
\label{eq:R-bellman-eq1}
\phi^*_s(\bw_s) = \sup_{\bpii_s} \Big\{\bw_s^\T \bpii_s + \phi_s(\bpii_s)\Big\}, \forall \bw_s\in \bbR^{|\cA|},    
\end{equation}
then  the value iteration can be updated as $\cT^{\RE,\bphi}[V](s) = \phi^*_s(\bw_s)$, where $\bw_s\in \bbR^{|\cA|}$ with elements $w_{sa} = r(a|s)+\sum_{s'}q(s'|s,a)V(s')$. 
%An optimal policy then can be computed as  $ \bpii^*_s = \text{argmax}_{\bpii_s} \{\bw_s^\T \bpii_s + \phi_s(\bpii_s)\}, \forall \bw_s\in \bbR^{|\cA|}$. 
%Under some settings, e.g., $\phi_s(\cdot)$ is the relative entropy function, there are closed form solutions to solve \eqref{eq:R-bellman-eq1}. Otherwise, the problems can still efficiently solved by convex optimization. 
If $\phi_s(\bpii_s) = -\sum_{a\in\cA} \pi(a|s)\ln \pi(a|s)$ (relative entropy), the model becomes the entropy-regularized one.
%, which is already popular in the RL community \cite{Fox2015taming,Schulman2015trust,Schulman2017equivalence,Haarnoja2017RL_ER,Haarnoja2018softa,Haarnoja2018softb}, with  
%various motivations such as improving exploration and compositionality in reinforcement learning. 
We call this MDP framework as \textbf{Entropy-Regularized MDP} (ER-MDP).

%Note that the contraction property holds for any function $\phi_s(\bpii_s)$, not necessarily concave. An equivalent formulation  of the Markov problem \eqref{prob:Regularized-MDPs} is 
%\begin{align}
%	\underset{\bpii}{\text{sup}}\qquad & \bbE_{\bpii} \left[\sum_{t=0}^{\infty} \gamma^{[t]} \phi_{s_t}(\bpii_{s_t}) \right] & \\
%	 \text{subject to} \qquad & \bbE_{\bpii} \left[\sum_{t=0}^{\infty} \gamma^{[t]}r(a_t|s_t) \right]  = \widetilde{\bbE}^R&  \nonumber\\
%	  &  \bpii\in\Delta^{\pi} &  \nonumber
%\end{align}
%where $\widetilde{\bbE}^R$ is an empirical expectation of the accumulated rewards. 

\textbf{{Stochastic MDPs (a.k.a. Dynamic Discrete Choice Modeling).} }
This framework has been developed and widely studied in the econometrics community for the problem of learning/predicting people behavior 
%inferring/estimating a reward/utility function 
from observations (a.k.a IRL) \citep[e.g.,][]{Rust1987GMC}. In this framework, it is assumed that the reward function is stochastic and can be written as  $\widetilde{r}(a|s) = r(a|s)+\epsilon(a|s)$, where $r(a|s)$ is deterministic and $\epsilon(a|s)$ is a random noise. 
Let $\beps_s  =\{\epsilon(a|s)|\ a\in \cA\}$.
We further assume that the distributions of $\beps_s$, $\forall s\in\cS$, are continuous and independent over states.
We define a sequence of decision rule $\bPsi = \{\bpsi^0,\bpsi^1,\ldots\}$, where $\bpsi^t$ is a set of decision rules at time $t$, i.e., $\bpsi^t = \{\psi(s_t|\beps_{s_t}),\ s_t\in\cS\}$, where $\psi(s_t|\beps_{s_t}):\bbR^{|\cA|}\rightarrow \cA$ is a deterministic policy at state $s_t$ under a realization of the random noise $\beps_{s_t}$. The aim is to select a decision rule  to maximize the expected discounted stochastic reward
\begin{align}
   \sup_{\bpsi^0,\bpsi^1,...}  \bbE_{\beps} \left\{ \left[  \bbE_{\tau \sim (\bPsi,\bq)}\Bigg[\sum_{t=0}^{\infty}\gamma^{[t]} \Big(r(a_t|s_t) +\epsilon(a_t|s_t)\Big) \Bigg] \right]\right\}.
\end{align}
To solve the above Markov problem, we also define the value function
\begin{align}
V^{\SR}(s) &=  \sup_{\bpsi^0,\bpsi^1,...} \bbE_{\beps} \left\{ \left[  \bbE_{\tau \sim (\bPsi,\bq)}\Bigg[\sum_{t=0}^{\infty}\gamma^{[t]} \Big(r(a_t|s_t) +\epsilon(a_t|s_t)\Big)\Big|\ s=s_0 \Bigg]\right] \right\},\ \forall a\in\cS.  \nonumber 
\end{align}
Then \cite{Rust1988MLE} shows that $V^{\SR}$ is a solution to the contraction mapping
$
 \cT^{\SR}[V](s) = \sup_{\bpsi_s}\{\bbE_{\beps_s}[z(s,a|V)] \}  
$
where 
$z(s,a|V)  = r(a|s)+ \epsilon(a|s) +\bbE[V(s')]$, 
and there is an memory-less optimal decision rule specified as
$
\psi(s|\beps_s) = \text{argmax}_{a\in \cA} \left\{ z(s,a,V^{\SR})\right\}.
$
Basically, a set of decision rule $\{\bpsi_s,s\in\cS\}$ can be viewed as a policy where  the probability  of making  action $a$ at state $s$ is also the density of making action $a$, i.e.,  $\pi(a|s)= \bbP_{\beps}(\phi(s|\beps_s) = a)$.  
%If we denote by $\Phi(\bw_s):\bbR^{|\cA|}\rightarrow 1$ a (social surplus) function  such that $\Phi(\bw_s) = \bbE_{\beps_s}[\max_a \{ w_{sa}+\epsilon(a|s)\}]$, then the density of making action $a\in\cA$ becomes
%$
%\pi(a|s) = \partial \Phi(\bw_s)/\partial w_{sa} 
%$, where $w_{sa} = r(a|s)+ %\sum_{s'}q(s'|s,a)V(s')$.
In particular, if $\epsilon(a|s)$ are i.i.d and follow the Gumbel distribution (or Extreme Value Type I distribution), we obtain a MDP framework so-called \textit{the logit-based dynamic discrete choice model} \citep{Rust1988MLE}. 
%To be in line with other MDP frameworks considered in this paper, 
We call this model as\textbf{ Stochastic EV  MDP} (or EV-MDP for short). %We will later on show that this framework is equivalent to the ER-MDP presented above.

\section{Distributional Stochastic  MDPs}
\label{sec:DS-MDP}
%\mtien{Provide some motivation for the model}
The Bellman update in a stochastic MDP model involves an expectation of the form $\bbE_{\beps_s}\left[\max_{a\in\cA} \{w_{sa}+\epsilon(a|s))\} \right]$, which is typically difficult to evaluate even when the distribution of $\beps_s$ is well specified \citep{steele1997probability}. To over come this, inspired by the persistent model studied in  \citep{Bertsimas2006persistence,Natarajan2009persistency},
we introduce a new MDP framework, called \textit{Distributional Stochastic MDP} (DS-MDP), by extending the S-MDP framework by assuming that  the distribution of the random terms $\beps$ is not given certainly, but belongs to a set  of distributions $\Xi$. 
This new setting makes the corresponding Markov problem easier to solve by convex optimization, under some classes of distributions with marginal information. We will show later this new framework is strongly related to the regularized MDP one. 

Under the ambiguity setting, we define a Markov problem where the objective  is to select a sequence of distributions $\bD = \{\bd^0,\bd^1,\ldots\}$ and a sequence of decision rules $\bPsi = \{\bpsi^0,\bpsi^1,\ldots\}$ in such a way that the following expected reward function is maximized.
\begin{align}
   \sup_{\substack{\bd^t \in  \Xi \\ t=0,1,\ldots} }\left\{ \sup_{\bpsi^0,\bpsi^1,...} \left\{  \bbE_{\substack{\beps^t \sim \bd^t\\ t=0,\ldots} } \left[  \bbE_{\tau \sim (\bpsi,\bq)}\Bigg[\sum_{t=0}^{\infty}\gamma^{[t]} \Big(r(a_t|s_t) +\epsilon(a|s)\Big) \Bigg] \right] \right\}\right\},\label{eq:DS-infinite-horizon-form}
\end{align}
where  $\beps$ the random noise vector  at time $t$. That is, at time step $t$, we select a 
distribution $\bd^t\in\Xi$ describing how the random noises $\beps$ vary and a
decision rule $\bpsi^t$ to select action according to each realization of the random noise vector. 
To make the Markov problem  tractable, we assume that
$\Xi$ is state-wise decomposable, i.e., $\Xi = \otimes_{s\in\cS}\Xi_s$, where $\Xi_s$ is an ambiguity set for the distributions of the noise vector $\beps_s$ at state $s$. This assumption allows to solve the Markov problem by value iteration or policy iteration, similarly to other MDP frameworks. 

We define the value function (expected accumulated discounted reward) $V^{\DS,\Xi}: \cS \rightarrow \bbR$ as 
\begin{equation}
\label{eq:DS-defineV}
V^{\DS,\Xi}(s) =  \sup_{\substack{\bd^t \in\Xi \\ t=0,\ldots}} \sup_{\bpsi^0,\bpsi^1,...} \left\{  \bbE_{\substack{\beps^t \sim \bd^t \\ t=0,\ldots}} \left[  \bbE_{\tau \sim (\bPsi,\bq)}\Bigg[\sum_{t=0}^{\infty}\gamma^{[t]} \Big(r(a_t|s_t) +\epsilon(a_t|s_t)\Big)\ \Bigg|s_0=s \Bigg] \right] \right\},    
\end{equation}
and the mapping
\begin{equation}
\label{eq:DS-mapping}
\cT^{\DS,\Xi}[V](s) = \sup_{d_s \in \Xi_s} \left\{\bbE_{\beps_s \sim d_s}\left[\max_{a\in\cA}\left\{ r(a|s)+ \epsilon(a|s) + \gamma\sum_{s'}q(s'|s,a)V(s')\right\} \right] \right\}.     
\end{equation}
The following theorem shows that basic theoretical properties that hold for other MDP frameworks also hold for the DS-MDP.
%, making the corresponding Markov decision problem tractable. 
\begin{theorem}
\label{th:DS-contractionmapping-optimal-policy}
For any nonempty set of distributions $\Xi$, we have the following properties 
\begin{itemize}
    \item[(i)] $\cT^{\DS,\Xi}[V]$ is monotonically increasing  and translation invariant, i.e., $\cT^{\DS,\Xi}[V] \geq \cT^{\DS,\Xi}[V']$ for any $V,V'\in\bbR^{|\cS|}, V\geq V'$, and $\cT^{\DS,\Xi}[V + \alpha \be] = \cT^{\DS,\Xi}[V] + \alpha$, where $\alpha$ is a scalar and $\be$ is a unit vector of size $|\cS|$. 
    \item[(ii)] $\cT^{\DS,\Xi}[V]$ is a contraction mapping of parameter $\gamma$ w.r.t. the infinity norm, i.e., for any $V,V'\in\bbR^{|\cS|}$ we have $||\cT^{\DS,\Xi}[V] - \cT^{\DS,\Xi}[V']||_\infty \leq \gamma ||V-V'||_\infty$. 
    \item[(iii)] If $\sup_{d_s\in\Xi}\bbE_{\beps_s\sim d_s}[\max_a\{r(a|s)+\epsilon(a|s)\}]$ are bounded from above for all $s\in\cS$, then  $V^{\DS,\Xi}$ is the unique fixed point solution to the system $\cT^{\DS,\Xi}[V] = V$. 
    \item[(iv)] The densities of the optimal decision rules (i.e. optimal policy) $\bpsi_s , \forall s\in \cS$, can be computed as
\[
\pi^{\DS,*} (a|s) =\partial \varphi_s(\bw_s)/\partial w_{sa}, \ \forall a\in\cA, s\in\cS
\]
where $\bw_s$ is a vector of size $|\cA|$ with entries $ w_{sa} = r(a|s)+ \gamma\sum_{s'}q(s'|s,a)V^{\DS,\Xi}(s')$ and  
\begin{align}
 \varphi_s(\bw_s) &= \sup_{d_s\in\Xi_s}\left\{\bbE_{\beps_s \sim d_s}[\max_a \{ w_{sa}+\epsilon(a|s)\}]\right\}.\label{eq:DS-expected-phi} 
\end{align}
\end{itemize}
\end{theorem}
Theorem \ref{th:DS-contractionmapping-optimal-policy}-(i) is quite obvious to validate. The contraction property (ii) can be done by  looking at the structure of the mapping \eqref{eq:DS-mapping} and (iii) can be validated by extending \eqref{eq:DS-mapping} recursively. The assumption in (iii) generally holds if the noise $\beps_s$ has a finite first moment for any distribution $d_s\in\Xi$ for any $s\in\cS$, i.e., $\sup_{d_s \in\Xi}\bbE_{\beps_s\sim d_s}[\beps_s]<\infty$. Given (ii)  and (iii), one can find an optimal decision rule by solving the Bellman equation \eqref{eq:DS-mapping}. Hence, we can make use of the properties of the semi semi-parametric choice model  propose by \cite{Natarajan2009persistency} to obtain the optimal decision rule (iv).
%We refer the reader to the appendix for the  detailed proof. 
Note that in the S-MDP framework, the optimal policy is defined similarly, but for only one distribution $d_s$. Let $d_s(\bw_s)$ denote the distribution or a limit of a sequence of distributions (in case that the set $\Xi_s$ is not compact) that attains the optimal value $\varphi_s(\bw_s)$, the optimal policy  can be expressed as
$
\pi^{\DS,*} (a|s) = \bbP_{\beps_s\sim d_s(\bw_s)}\left[a = \text{argmax}_{a\in\cA}\{w_{sa}+\epsilon(a|s)\} \right].
$
%We also see that the expected reward function $\varphi_s(\bw_s)$ is monotonically increasing in $\bw_s$, i.e., for any $\bw^1\geq \bw_2$ we have $\varphi_s(\bw^1_s) \geq \varphi_s(\bw^2_s)$. Moreover, it also satisfy a ``\textit{Translation Invariance}'' property, i.e., $\varphi_s(\bw_s +t\be) = \varphi_s(\bw_s) + t$.  \cite{Geist2019theory} also show that  these properties also hold for the Legendre-Fenchel transform  function $\phi^*_s(\bw_s)$ in the R-MDP framework presented above. In the next section, we will show that, for any  convex function $\phi_s$ there is a set of distributions $\Xi_s$ such that the two functions $\varphi_s(\bw_s)$ and $\phi^*_s(\bw_s)$ are identical (and vice-versa), leading to the equivalence between the R-MDP and DS-MDP framework.
It is also interesting to look at a relevant version of the DS-MDP, where we take an \textit{inf} over the distribution set $\Xi$, instead of \textit{sup} as in \eqref{eq:DS-infinite-horizon-form}, i.e., $ \inf_{\bd^0,\bd^1,\ldots } \sup_{\bpsi^0,\bpsi^1,...}\{\cdot\}$. In this case, an optimal decision rule  is more difficult to identify, even for one-stage problems, thus less interesting \citep{Natarajan2009persistency}.   

%Another relevant question here is how to define the ambiguity sets of distributions $\Xi$ in such a way that the optimal policy can be computed in a tractable way. We will discuss this issue in the next section, where we show that under some ambiguity set with  marginal information, the corresponding Markov decision problems can be convert to R-MDP ones with (strictly) concave regularizers, which can be solved efficiently in polynomial time. 

\section{Nested Relation between Stochastic and Regularized MDP Frameworks}
\label{sec:nested-relation}
In this section we analyze the relation between the stochastic and regularized MDP frameworks mentioned above. First,  we summarize all the MDP frameworks considered in Tab. \ref{tab:MDPs}. The constrained MDP one on the last row will be discussed in the next section.   

\begin{minipage}{\textwidth}
  \begin{minipage}[b]{0.49\textwidth}
    \centering
   \begin{tabular}{l|l}
MDP frameworks                  & Abbreviation  \\ 
\hline
Standard MDP                    & MDP          \\
Stochastic EV  MDP        & EV-MDP       \\
Stochastic  MDP           & S-MDP        \\
Distributional Stochastic MDP & DS-MDP       \\
Regularized MDP                 & R-MDP        \\
Entropy Regularized MDP         & ER-MDP       \\
Constrained MDP                 & CT-MDP      
\end{tabular}
      \captionof{table}{Summary of MDP frameworks \label{tab:MDPs} }
    \end{minipage}
  \hfill        
  \begin{minipage}[b]{0.49\textwidth}
    \centering
    \includegraphics[width=0.6\textwidth]{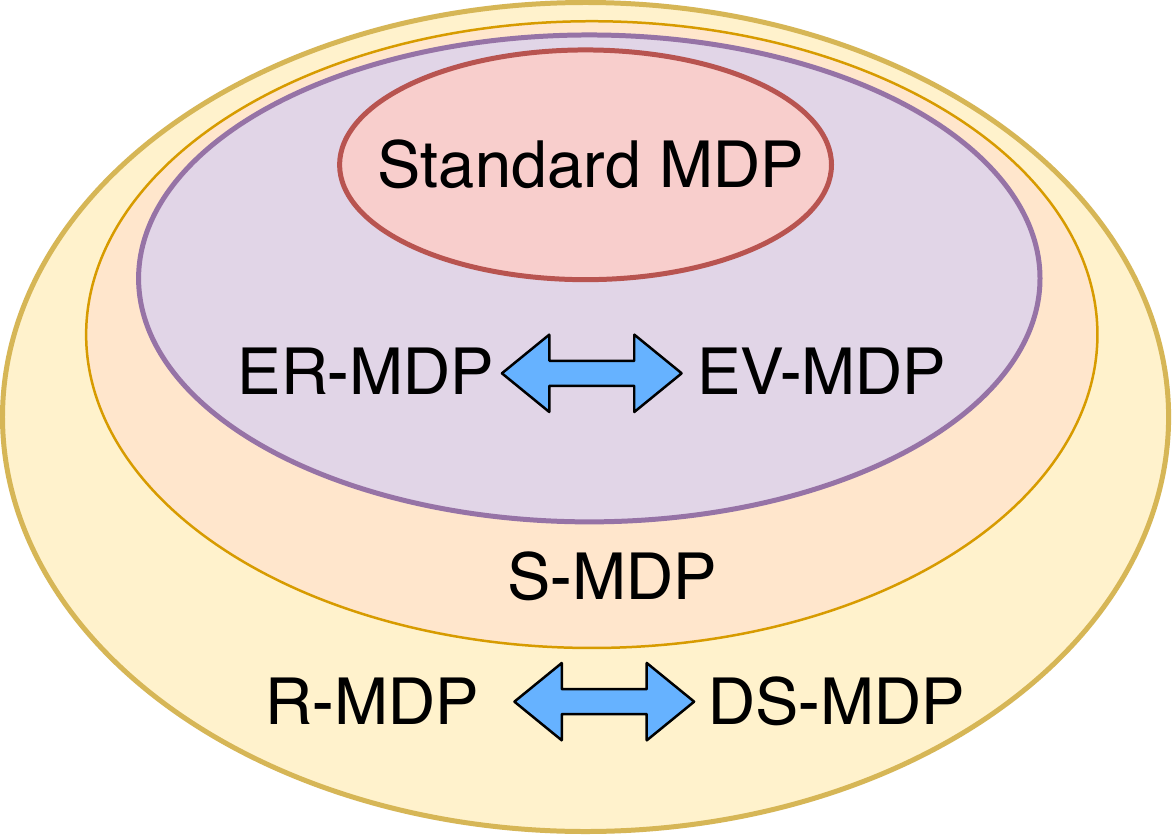}
    \captionof{figure}{Nested relation between MDP frameworks  \label{fig:relation}}
  \end{minipage}
  \end{minipage}

\iffalse
\begin{table}[htb]
\centering
\begin{tabular}{l|l}
MDP frameworks                  & Abbreviation  \\ 
\hline
Standard MDP                    & MDP          \\
Stochastic EV  MDP        & EV-MDP       \\
Stochastic  MDP           & S-MDP        \\
Distributional Stochastic MDP & DS-MDP       \\
Regularized MDP                 & R-MDP        \\
Entropy Regularized MDP         & ER-MDP       \\
Constrained MDP                 & CT-MDP      
\end{tabular}
\caption{Summary of MDP frameworks}
\label{tab:MDPs}
\end{table}

\begin{figure}[htb]
    \centering
    \includegraphics[width=0.4\linewidth]{MDP9.pdf}
    \caption{Nested relation between MDP frameworks}
    \label{fig:relation}
\end{figure}
\fi

In general, we will validate the  diagram in Fig. \ref{fig:relation}, which says that
 (i) the R-MDP are equivalent to DS-MDPs and strictly covers other frameworks, 
(ii) the ER-MDP is equivalent to the S-MDP with Gumbel distributed random noises  and both are strictly covered by the S-MDP.
This diagram shows a broad and insightful connection/relation between  regularized and stochastic MDP frameworks, thus provides new ways to interpret each MDP framework through others. To facilitate our exposition, we need the following definitions. 

%\begin{definition}
%For any MDPs \textbf{X}, let $\bV^*[\textbf{X},\br]$ be the value function and $\bPii^*[\textbf{X},\br]$ be the
%set of optimal policies given by \textbf{X} with reward function $\br = \{r(a|s),\forall a\in\cA, s\in\cS\}$. 
%\end{definition}
\begin{definition}[Equivalent MDP models]
\label{def:equip-MDP}
Two MDP models $\textbf{X}$ and $\textbf{Y}$  are said to be ``\textit{equivalent}'', denoted by $X\equiv Y$, if and only if 
they are specified by the same tuple  $(\cS,\cA,\bq, \gamma)$ and there are a constant vector $\bap$ of size ($|\cS|\times|\cA|$) such that for any reward functions $\br^{X},\br^{Y}\in \bbR^{|\cS|\times|\cA|}$, $\br^X- \br^Y = \bap$, model $\bX$ with rewards $\br^X$ and model $\bY$ with rewards $\br^Y$ always yield the same the value function  and  a common stationary optimal policy $\{\pi^*(a|s),\ \forall a,s\}$.
\end{definition}
The equivalence property between two MDP models  implies that the two models are identical to use, in both perspectives of RL and IRL. Under this definition,  two ER-MDP models based on a relative entropy regularizer $\eta_1 \sum_{a}\pi(a|s)\ln(\pi(a|s))$ and a KL divergence $\eta_2\KL(\bpii_s||\overline{\bpii}_s)$  are equivalent if $\eta_1=\eta_2$, but not equivalent if $\eta_1 \neq \eta_2$.  

We can view a MDP framework as a set of MDP models, e.g., the S-MDP framework contains S-MDP models, each is specified by a distribution set $\{d_s,\ s\in\cS\}$ of the random noises $\{\beps_s,\ s\in\cS\}$.
We define ``equivalent'' MDP frameworks as follows. 
\begin{definition}[Equivalent MDP frameworks]
\label{def:equip-MDP-FR}
Two MDP frameworks $\cX$ and $\cY$  are said to be ``\textit{equivalent}'' if and only if for any MDP model $\bX\in\cX$, there is a MDP model $\bY \in \cY$ such that $\bX\equiv \bY$, and vice-versa. 
\end{definition}
In other words, the ``\textit{equivalence}'' between two MDP frameworks implies that they share the same set of equivalent MDP models, and the uses of these MDP frameworks are indeed identical. We also define when a framework is strictly subsumed by another one.
\begin{definition}[Subsumed MDP framework]
Given two MDP frameworks $\cX$ and $\cY$, we say ``$\cX$ strictly subsumes $\cY$'' if  and only if for any MDP model $\bY\in\cY$ there is always $\bX\in\cX$ such that $\bX \equiv \bY$, but there is $\bX \in \cX$ such that there is no $\bY \in \cY$ such that $\bX$ and $\bY$ are equivalent.  
\end{definition}
So if a MDP framework $\cX$ subsumes another $\cY$, it means $\cX$ is more flexible to use. It is clear that the R-MDP subsumes the EV-MDP and standard MDP, and the DS-MDP subsumes the S-MDP and EV-MDP ones (but not sure the subsumption is strict). The question now is what is the relation between two modeling approaches: \textit{stochastic} and \textit{regularized}. We start investigating this by the following proposition.
\begin{proposition}
\label{th:Ev-ER-MDPs}
The ER-MDP and EV-MDP frameworks are equivalent and strictly subsume the standard unconstrained MDP one. 
\end{proposition}
The equivalence between ER-MDP and EV-MDP  in the finite-horizon case has been shown in \cite{Ermon2015learning} and can be validated easily by looking at the formulations used to compute optimal policies in previous studies  \citep{ziebart2010_IRL_Causal, Rust1988MLE}. The claim that ER-MDP and EV-MDP strictly subsumes the standard one is also easy to verify, using the fact the standard MDP typically yields deterministic policies while the other frameworks give randomized and soft optimal policies. 

%The theorem  shows the equivalence between two popular MDP frameworks in the machine learning and econometrics literature, thus  gives a new way to interpret the ER-MDP framework through the lends of stochastic rewards and vice-versa. 
There is a well-known similarity between Shannon's entropy \cite{Shannon1948mathematical} and the multinomial logit model \citep{Anas1983discrete}. 
%i.e., the (unique) solution to the entropy minimization problem $\min_{x\in\Delta (N)} \sum_{i=1}^N r_i x_i+ x_i \ln x_i$  is equal to the probabilities $\bbP_{\beps}[i = \text{argmax}_{i} \{r_i+\epsilon_i\} ]$, where $\Delta(N)$ is the simplex in $\bbR^{N}$ and  $\beps = \{\epsilon_i,i = 1,\ldots, N \}$ are i.i.d Gumbel distributed. 
Proposition \ref{th:Ev-ER-MDPs} extends this similarity to a multi-stage decision-making setting. It also indicates a  connection between  EV-MDPs and the maximum causal entropy principle \citep{ziebart2010_IRL_Causal}, which implies that an optimal decision rule given by an EV-MDP model will maximize a causal entropy function and satisfy a prediction log-loss guarantee.   

An application of the ER-MDP framework is to add a KL divergence  $\phi_s(\bpii_s) = -\eta\KL(\bpii_s||\overline{\bpii}_s) = \eta\sum_{a}\pi(a|s)\ln \frac{\pi(a|s)}{\overline{\pi}(a|s)}$  to control the divergence between a desired optimal policy and a given policy
%for instance, to prevent early convergence to sub-optimal policies in a policy iteration algorithm 
\cite{Haarnoja2017RL_ER,Abdolmaleki2018maximum}. According to Theorem \ref{th:Ev-ER-MDPs}, this ER-MDP model is \textit{equivalent} to a stochastic MDP model with rewards $r'(a|s) = r(a|s) -\eta \overline{\pi}(a|s)$ and Gumbel distributed i.i.d random noises. In other  words, by adding stochastic noises to the reward function, we can penalize the divergence between two policies.   

We now move to the S-MDP framework and have Proposition \ref{prop:S-MDP-subsume-EV-MDP}. It is obvious to see that the EV-MDP is just special case of the S-MDP when the random noises follow the Gumbel distribution. Thus S-MDP subsumes both the EV-MDP and ER-MDP. However, it is not straightforward to see that the S-MDP \textit{strictly} subsumes the other frameworks, as one can show that an ER- or EV-MDP model can yield any optimal optimal policy in $\Delta^\pi$ when varying the reward function. 
We do this by giving a counter example where the random noises in the S-MDP model are no-longer i.i.d.  
\begin{proposition}
\label{prop:S-MDP-subsume-EV-MDP}
The S-MDP strictly subsumes the EV-MDP and ER-MDP frameworks.
\end{proposition}

We now look at broader frameworks. The R-MDP is the most general frameworks among regularized ones,  and the DS-MDP covers all other stochastic MDP frameworks. Interestingly, Theorem \ref{th:R-and-DS-MDP} below tells use that the two general frameworks are actually equivalent. The proof makes use of  Theorem \ref{th:DS-contractionmapping-optimal-policy}  and some results  in \cite{Feng2017relationDC}. The equivalence  provides  a new way to view regularized MDPs through the lens of stochastic MDPs. 
That is, adding (concave) regularizers to the MDP model is equivalent to assuming that the reward function contains some random noises of ambiguous distribution. This also shows the robustness of the R-MDP framework from a stochastic-reward viewpoint.

\begin{theorem}
\label{th:R-and-DS-MDP}
The R-MDP and DS-MDP frameworks are equivalent and strictly subsume the S-MDP framework if $|\cA|\geq 3$.
\end{theorem}

We also see that, given an ambiguity set of continuous distributions $\Xi$, the DS-MDP model is equivalent to a R-MDP model with regularizers $\phi_s(\bpii_s) = -  \varphi^*_s(\bpii_s)$, where $\varphi^*_s(\bpii_s)$ is the convex conjugate of function $\varphi_s(\bw_s)$ defined in \eqref{eq:DS-expected-phi}, i.e., $\phi_s(\bpii_s) = -\varphi^*_s(\bpii_s) = -\sup_{\bw_s \in \bbR^{|\cA|}} \{\bpii^\T_s \bw_s - \varphi_s(\bw_s) \}$. Under some ambiguity settings with marginal information, one can obtain a closed form to compute the convex conjugate functions $\varphi^*_s(\bpii_s)$ and we provide some examples in Proposition \ref{prop:ex-DS-R-MDP} below, 
%(Section \ref{sup:sec:Examples-DS-MDP}), 
which may suggest new regularizers %that can be added to the reward function
with interpretations from a stochastic-rewards viewpoint. 
Proposition \ref{prop:ex-DS-R-MDP} below provides some examples of DS-MDP models and their equivalent R-MDP ones. These examples extend some results developed for the   persistent model \cite{Natarajan2009persistency} to our MDP settings.  

\begin{proposition}
\label{prop:ex-DS-R-MDP}
The DS-MDP models on the left column are equivalent to the corresponding R-MDP models on the right column. 
\begin{table}[htb]\small
\centering
\begin{tabular}{l|l|l}
Case&DS-MDP, for any $s\in\cS$                & R-MDP, for any $s\in\cS$  \\ 
\hline
(i)&$\Xi_s = \Bigg\{d_s\Big| \epsilon(a|s)\sim F_{sa}(\cdot), \forall a\in\cA\Bigg\}$             & $\phi_s (\bpii_s) = \sum_{a\in\cA} \int_{1-\pi(a|s)}^1 F_{sa}^{-1}(t) dt$         \\\hline
(ii)&$\Xi_s = \left\{d_s
  \middle\vert \begin{array}{@{}ll@{}} \text{dev}(\epsilon(a|s))  = \sigma(a|s) \\ \bbE[\epsilon(a|s)]=0\end{array} \right\}$      & $\phi_s (\bpii_s) = \sum_{a\in\cA} \sigma(a|s)\sqrt{\pi(a|s)(1-\pi(a|s))}$       \\\hline
(iii)&$\Xi_s = \left\{d_s
  \middle\vert \begin{array}{@{}ll@{}} \text{Cov}(\beps_s)  = \Sigma_s \\ \bbE[\beps_s]=0\end{array} \right\}$               & $\phi_s (\bpii_s) = \text{trace}\Big( \Sigma_s^{1/2}\big(\text{Diag}(\bpii_s) - \bpii_s\bpii_s^\T \big)\Sigma_s^{1/2} \Big)^{1/2}$        \\\hline
\end{tabular}
\end{table}
where $F_{as}(\cdot)$ is the marginal  (continuous) cumulative distribution function of $\epsilon(a|s)$, $\text{dev}(\epsilon(a|s))$ is the marginal standard deviation of $\epsilon(a|s)$, $\text{trace}(\bA)$ is the trace of matrix $\bA$, $\text{Diag(\bv)}$ is the  diagonal matrix with the elements of vector $\bv$ on the main diagonal, and 
$\text{Cov}(\beps_s)$ is the covariance matrix of the random vector $\beps_s$. Here we assume that $\Sigma_s$ is positive definite.
\end{proposition}

For (i) and (ii), the models are based on the the marginal distribution  and marginal  moment models considered  \cite{Natarajan2009persistency}, in which only the cumulative distribution, or the first and moments of the marginal distributions of the random noises are known. 
Example (iii) concerns models in which only the covariance information of the distribution in $\Xi$ is known and the equivalent R-MDP model is derived using   \cite{mishra2014theoretical}.

%-----------------------------------
%
%-----------------------------------

\section{Connection to Constrained MDPs}
\label{sec:CT-MDP}
We investigate a relation between the  MDP frameworks considered above and the constrained MDP (CT-MDP) one, an extension of the standard MDP where constraints are imposed on the policy. A typically CT-MDP model has constraints  on a long-term performance of the policy \citep{Altman1999constrained}.
%, and \cite{Altman1999constrained} show that, under some \textit{linearity} setting, the constrained Markov problem can be solved by linear programming. 
In this section, we only consider constraints on state policies, i.e., the policy variables at each state belong to a feasible set, and these sets are independent over states.
%In other words, we assume that the feasible set of the policies satisfies a rectangularity property. 
This allows to solve the constrained MDP problem by value iteration or policy iteration as in the case of standard MDPs. 

%We first present a formulation  for the constrained MDP (CT-MDP) framework and show how it is related to other ones. The relation also provides a new way to view optimal policies given by other types of MDP frameworks.
\subsection{Constrained MDPs} The constrained version requires that the policy $\bpii$ belong to a set 
$
\bpii \in \cD^{\pi} \subset \Delta^{\pi}. 
$
We further assume that $\cD^\pi$ is state-wise decomposable, i.e., 
$\cD^{\pi} = \otimes_{s\in\cS}\cD^{\pi}_{s}$, where $\cD^{\pi}_{s} \subset \Delta(|\cA|)$ is an support set  for the policy $\{\pi(a|s),\ a\in\cA\}$, for any $s\in \cS$. We further assume that $\cD^{\pi}_s$ is nonempty. 
The Markov decision problem becomes
\begin{align}
\sup_{\substack{\bpii^t \in \cD^\pi \\t=0,\ldots }}\Bigg\{ \bbE_{\tau \sim (\bPi,\bq)}\Bigg[\sum_{t=0}^{\infty}\gamma^{[t]} r(a_t|s_t)  \Bigg]\Bigg\}\nonumber
\end{align}
We can then define a value function $V^{\CC,\cD^\pi}$  and Bellman equation $\cT^{\CT,\cD^\pi}[V]$ 
analogously to the standard MDP, except that the policy $\bpii$ are required to lie in the set $\cD^\pi$. Note that the Bellman update in this case $\cT^{\CT,\cD^\pi}$ is contraction for any non-empty set $\cD^{\pi}$.

\subsection{CT-MDP and  other MDP frameworks}
We first note that the CT-MDP covers the standard MDP framework when $\cD^\pi_s$ is the simplex in $\bbR^{|\cA|}$. 
Even-though the CT-MDP looks general, as we can introduce any feasible set $\cD^\pi_s$,  Theorem \ref{th:CT-MDP-and-others} below shows that, in the context of Definitions  \ref{def:equip-MDP} and  \ref{def:equip-MDP-FR}, the CT-MDP framework is not subsumed or cover any stochastic or regularized MDP  frameworks considered above, except the standard one. 
On the other hand, the  R-MDP and DS-MDP frameworks are the most general among those presented in Fig. \ref{fig:relation}, but cannot cover the CT-MDP. 
\begin{theorem}
\label{th:CT-MDP-and-others}
There is an ER-MDP model $\bX$ with parameter $\eta>0$ such that there is no CT-MDP model that is equivalent to $\bX$. On the other hand, there is also a CT-MDP model $\bY$ such that there is no R-MDP model being equivalent to $\bY$.
\end{theorem}
%As a result, we have Figure \ref{fig:relation-CTMDP} illustrating the relation between the CT-MDP framework and other MDP frameworks considered above. In general, it tells us that the CT-MDP framework covers the standards MDP one but cannot cover any other MDP frameworks considered in the previous section. 

\iffalse 
\begin{figure}[htb]
    \centering
    \includegraphics[width=0.5\linewidth]{CT-MDP1.pdf}
    \caption{CT-MDP and other MDP frameworks}
    \label{fig:relation-CTMDP}
\end{figure}
\fi

However, if we fix the reward function, then we can show that, even though the CT-MDP framework does not cover the R-MDP one, but for any R-MDP model, there is a CT-MDP model specified by convex constraints that yields an identical optimal policy. 
%This provides a new way to interpret the MDP frameworks presented in the previous section through the CT-MDP one.  
\begin{proposition}
\label{prop:ER-to-CT}
 Given a R-MDP model defined by a tuple $(\cS,\cA,\bq, \br,\gamma)$   and the set of strictly concave functions $\{\phi_s|\ s\in\cS\}$, there exists a set of constants $\{c_s,\ s\in\cS\}$ such that the CT-MDP model defined by the same tuple but with a shifted reward function $\br'_s = \br_s - c_s$ and feasible sets $\cD^{\pi}_s = \{\bpii_s\in \Delta(|\cA|)|\ -\phi_s(\bpii_s) \leq c_s\}$ will yield an identical optimal policy.
\end{proposition}

It is important to note that $\{c_s,\, \forall s\}$ specified in  Proposition \ref{prop:ER-to-CT} depend on the reward vector $\br$, thus the CT-MDP model cannot yield identical optimal policies for any reward function. The reverse is also true, that is, for any CT-MDP model specified by a set of convex constraints, we always can find a R-MDP model that yield an identical optimal policy (see 
%Section \ref{sup:sec:from-R-MDP-to-CT-MDP} in 
the supplementary material for details).
%Propositions \ref{prop:CT-to-RE} shows how to represent a stochastic or regularized MDP model by a CT-MDP. We believe these results provide an unified way to interpret different MDP frameworks via constrained policies and we discuss this in detail in the following section. 

\subsection{Discussion}
In general, we have seen that a stochastic MDP model can be interpreted as a regularized one and vice-versa. Proposition \ref{prop:ER-to-CT} shows that any R-MDP or DS-MDP model can be always represented by a CT-MDP one. In other words,  adding regularizers or random noises to the reward function is \textit{equivalent} to constraining the desired optimal policies. Let us illustrate this by some examples. 

A popular  and useful regularizer in RL is a KL divergence $-\eta \KL(\bpii_s||\overline{\bpii}_s)$  between two consecutive  policies.
%, to prevent early convergence to sub-optimal policies in a, for instance, policy iteration algorithm. 
We have shown that this regularized model is equivalent to a stochastic model with shifted rewards and Gumbel distributed random noises. From Proposition \ref{prop:ER-to-CT}  we see that these ER-MDP and EV-MDP models can be equivalently represented by a CT-MDP with constraints $\KL(\bpii_s||\overline{\bpii}_s) \leq c_s/\eta$, where $c_s \geq 0$ for all $s\in \cS$. Clearly, if $c_s\rightarrow 0$ then the optimal policy should converge to  $\overline{\bpii}_s$ and if we increase $c_s$, the feasible set $\cD^\pi_s$ will converge to $\Delta(|\cA|)$ (i.e. the simplex in $\bbR^{|\cA|}$) and we will retain a deterministic policy as in the unconstrained MDP model.

Several RL and IRL algorithms make use of soft policies resulting from relative entropy regularizers $\phi_s(\bpii_s) =-\eta \sum_{a}\pi(a|s)\ln \pi(a|s)$ \citep{ziebart2010_IRL_Causal}. Let $\widetilde{\bpii}^*_s = \text{argmax}_{\bpii_s \in\Delta(|\cA|)} \phi_s(\bpii_s)$, we see that (i) $\widetilde{\bpii}^*_s = \be/|\cA|$ (a random walk, $\be$ is a unit vector of size $|\cA|$) and (ii) for any $\bpii_s$ 
\[
\KL(\bpii_s||\widetilde{\bpii}^*_s) = \sum_{a\in\cA}\pi(a|s)\ln \pi(a|s) - \ln |\cA| = -\phi_s(\bpii)/\eta - \ln |\cA|.
\]
Thus the constraint $-\phi_s(\bpii_s) \leq c_s$ becomes $\KL(\bpii_s||\widetilde{\bpii}^*_s) \leq c_s/\eta - \ln |\cA|$. So, we see that by adding the relative entropy terms and varying the parameter $\eta$, we are controlling the KL divergence between a desired policy and the random walk vector $\widetilde{\bpii}^*_s$. Making the optimal policy closer to the random walk indeed helps improve exploration.% and smooth the optimal policies. %, as indicated in previous studies.

This analysis can be extended to a general R-MDP model with (strictly) concave regularizers $\{\phi_s,s\in\cS\}$. Also let $\widetilde{\bpii}^*_s = \text{argmax}_{\bpii_s \in\Delta(|\cA|)} \phi_s(\bpii_s)$, which are always uniquely determined. We can show that if $\widetilde{\bpii}^*_s$ contain non-zero elements, then we can write $\nabla \phi_s(\widetilde{\bpii}^*_s) = \alpha \be$, where $\alpha$ is a scalar.  The Bregman divergence ($\BD$) generated by $-\phi_s(\cdot)$ between two policies becomes
\[
\BD^{-\phi_s}(\bpii_s||\bpii^*_s) = -\phi_s(\bpii_s) + \phi_s(\widetilde{\bpii}^*_s) + \alpha - \alpha \be^\T \widetilde{\bpii}^*_s.
\]
Thus, the constraints $-\phi_s(\bpii_s) \leq c_s$ becomes $\BD^{-\phi_s}(\bpii_s||\bpii^*_s) \leq c_s + \phi_s(\widetilde{\bpii}^*_s) + \alpha - \alpha \be^\T \widetilde{\bpii}^*_s$. This generally implies that using a regularized MDP model is equivalent to imposing a Bregman divergence requirement on the desired policies.

The idea of using KL or Bregman divergence to control the distance between two consecutive policies in a policy iteration is interesting and useful \citep{Azar2012dynamic,Abdolmaleki2018maximum,Abdolmaleki2018relative,Mankowitz2019robust,Lee2018sparse,Chow2018path}. The CT-MDP framework provides a flexible and general way to implement it. For example, one can consider other types distance measures at a greedy step, i.e., we solve a constrained Bellman update with feasible sets defined as $\cD^\pi_{s} = \{\bpii_s\in\Delta(|\cA|),\ ||\bpii_s - \bpii^k_s||\leq \xi\}$, where the norm $||\cdot||$ can be  L1, L2 or L-infinity norm. If the norm is L1 or L-infinity, we greedy step can be efficiently performed by solving a linear program. If the norm is L2, then the greed step is still  a convex optimization problem, which is generally easy to solve. We can incorporate several previous policies to have a more robust greedy step, i.e., $\cD^\pi_{s} = \{\bpii_s\in\Delta(|\cA|),\  \max_{i\in \cK}||\bpii_s - \bpii^i_s||\leq \xi\}$, where $\cK$ is a set of indexes of previous policies considered. The constraints require that a desired policy after the greedy step should not be \textit{too far} from a set of some policies computed from previous iterations.
Moreover, different types of distance measures can be combined, e.g., KL divergence or different norms,  while keeping the constrained Bellman update tractable. We provide some computational complexity analyses for such constrained models  in the supplementary material.
%Note that the threshold $\xi$ may vary  over iterations to adaptively control the search region of the greedy steps. 

%The CT-MDP also has an advantageous feature that cannot be achieved explicitly by other frameworks, i.e., controlling the  sparsity of the MDPs. We refer the reader to the supplementary material for details.

\section{Conclusion}
\label{sec:conclude}
In this paper we study the relation between stochastic, regularized and constrained MDP frameworks. We first propose the DS-MDP by assuming the distribution of the random noises in the S-MDP is ambiguous and we validate some basic properties of the new framework, which allows to solve a DS-MDP problem by value iterations. We show that there is a nested relationship between stochastic and regularized MDP frameworks, which provides new ways to interpret regularized MDP frameworks through the lens of stochastic ones and vice-versa. We also show a strong connection between stochastic/regularized MDPs and constrained MDPs, and show how this finding would be useful in some well-known algorithmic schemes.   

\clearpage
\iffalse
\section*{Broader Impact}
We study the relation between different well-known MDP frameworks, so our work would be of interest for people in the RL community. This is a theoretical work, so we do not foresee any direct negative/positive impact in the society. However, since our work is directly related to RL algorithms, and we all know that RL have several impactful real-world applications that would present both positive and negative consequences. For example, RL becomes appealing in some  applications such as robust controlling and autonomous driving. However, if a RL algorithm is not carefully designed, it may yield bad decisions, which would further lead to serious consequences.\fi       
%In this paper we study a robust version of 

%Authors are required to include a statement of the broader impact of their work, including its ethical aspects and future societal consequences. 
%Authors should discuss both positive and negative outcomes, if any. For instance, authors should discuss a) 
%who may benefit from this research, b) who may be put at disadvantage from this research, c) what are the consequences of failure of the system, and d) whether the task/method leverages
%biases in the data. If authors believe this is not applicable to them, authors can simply state this.

%Use unnumbered first level headings for this section, which should go at the end of the paper. {\bf Note that this section does not count towards the eight pages of content that are allowed.}

\clearpage

\bibliographystyle{plainnat}
\bibliography{refs,ROMDP}

%\end{document}

\clearpage

\appendix
\section{Proofs of the Main Results}
%-------------------------------
%
%-------------------------------
\subsection{Proof of Theorem \ref{th:DS-contractionmapping-optimal-policy}}
\textbf{For} \textbf{(i)}, we see that, for any $V,V'\in\bbR^{|\cS|}$ such that $V\geq V'$
\begin{align}
\max_{a\in\cA} &\left\{ r(a|s)+ \epsilon(a|s) + \gamma\sum_{s'}q(s'|s,a)V(s')\right\} \nonumber \\
& \qquad \geq \max_{a\in\cA}\left\{ r(a|s)+ \epsilon(a|s) + \gamma\sum_{s'}q(s'|s,a)V'(s')\right\}, \:\forall s\in\cS,
\end{align}
thus $\cT^{\DS,\Xi}[V](s)\geq \cT^{\DS,\Xi}[V'](s)$ for any $s\in\cS$. The translation variance property is also easy to validate, as for any $s\in \cS$
\begin{align}
\cT^{\DS,\Xi}[V + \alpha \be](s) &=  \sup_{d_s \in \Xi_s} \left\{\bbE_{\beps_s \sim d_s}\left[\max_{a\in\cA}\left\{ r(a|s)+ \epsilon(a|s) + \gamma\sum_{s'}q(s'|s,a)V(s') + \alpha\right\} \right] \right\}   \nonumber \\
&=\sup_{d_s \in \Xi_s} \left\{\bbE_{\beps_s \sim d_s}\left[\max_{a\in\cA}\left\{ r(a|s)+ \epsilon(a|s) + \gamma\sum_{s'}q(s'|s,a)V(s') \right\} \right] \right\} + \alpha\nonumber \\
&= \cT^{\DS,\Xi}[V](s) +\alpha. \nonumber
\end{align}

For the \textbf{contraction property (ii)}, give any $V^1, V^2 \in\bbR^{|\cS|}$ and a state $s\in\cS$, in analogy to a standard proof, we consider two cases: $\cT^{\DS,\Xi}[V^1](s)\geq \cT^{\DS,\Xi}[V^2](s)$  and show that $\cT^{\DS,\Xi}[V^1](s)\geq \cT^{\DS,\Xi}[V^2](s) \leq ||V^1-V^2||_\infty$, the other case can be done in a similar way. Now for any $\xi_1>0$, let us choose a distribution $\overline{d}^1_s \in  \Xi_s$ such that
\[
\cT^{\DS,\Xi}[V^1](s) \leq     \bbE_{\beps_s \sim\overline{d}^1_s}\left[\max_{a\in\cA}\left\{ r(a|s)+ \epsilon(a|s) + \gamma\sum_{s'}q(s'|s,a)V^1(s')\right\} \right]  + \xi_1 
\]
For the sake of simplicity, we denote 
\begin{align}
z^1_s(a|\beps_s) = r(a|s)+ \epsilon(a|s) + \gamma\sum_{s'}q(s'|s,a)V^1(s') \nonumber\\
z^2_s(a|\beps_s) = r(a|s)+ \epsilon(a|s) + \gamma\sum_{s'}q(s'|s,a)V^2(s') \nonumber
\end{align}
For any fixed $\beps_s$, we write 
\begin{align}
&|\max_a \{z^1_s(a|\beps_s)\} - \max_a \{z^2_s(a|\beps_s)\}|\nonumber\\
&=\max\left\{\max_a \{z^1_s(a|\beps_s)\} - \max_a \{z^2_s(a|\beps_s); \max_a \{z^2_s(a|\beps_s)\} - \max_a \{z^1_s(a|\beps_s)\right\} \nonumber\\ 
&\stackrel{(i)}{\leq}  \max\{z^1_s(a_1|\beps_s) -  z^2_s(a_1|\beps_s);\; z^2_s(a_2|\beps_s) -  z^1_s(a_2|\beps_s)\}\nonumber\\
&=\gamma \max\left\{\sum_{s'} q(s'|s,a_1) (V^1(s')-V^2(s'));\; \sum_{s'} q(s'|s,a_2) (V^2(s')-V^1(s'))\right\}\nonumber\\
&\leq \gamma ||V^1-V^2||_\infty\label{eq:DS-eq1}
\end{align}
where in (i), $a_1 = \text{argmax}_{a\in\cA}\left\{z^1(a|\beps_s)\right\}$ and $a_2 = \text{argmax}_{a\in\cA}\left\{z^2(a|\beps_s)\right\}$. Now we evaluate $\cT^{\DS,\Xi}[V^1](s) - \cT^{\DS,\Xi}[V^2](s)$ as follows
\begin{align}
    &\cT^{\DS,\Xi}[V^1](s) - \cT^{\DS,\Xi}[V^2](s)  \nonumber \\
    &\leq \bbE_{\beps_s \sim\overline{d}^1_s}\left[\max_{a\in\cA}\left\{z^1(a|\beps_s)\right\}\right] - \bbE_{\beps_s \sim\overline{d}^1_s}\left[\max_{a\in\cA}\left\{z^2(a|\beps_s)\right\} \right] + \xi_1\nonumber\\
    &\leq \int_{\beps_s \sim \overline{d}^1_s} \left|\max_{a\in\cA}\left\{z^1(a|\beps_s)\right\} -\max_{a\in\cA}\left\{z^2(a|\beps_s)\right\}\right| f_s(\beps_s)d(\beps_s) + \xi-1\nonumber \\
    &\stackrel{(ii)}{\leq}  \int_{\beps_s \sim \overline{d}^1_s} \gamma||V^1-V^2||_\infty f_s(\beps_s)d(\beps_s) +\xi_1 = \gamma ||V^1-V^2||_\infty +\xi_1\nonumber,
\end{align}
where $f_s(\cdot)$ is the probability density function of $\beps_s$  and (ii) is due to \eqref{eq:DS-eq1}.
Since $\xi_1$ can be chosen arbitrarily small, we let $\xi_1\rightarrow 0$ and obtain 
$\cT^{\DS,\Xi}[V^1](s) - \cT^{\DS,\Xi}[V^2](s) \leq ||V^1-V^2||_\infty$.
In a similar manner, we can show that if $\cT^{\DS,\Xi}[V^1](s)< \cT^{\DS,\Xi}[V^2](s)$ then we also have  $\cT^{\DS,\Xi}[V^2](s) - \cT^{\DS,\Xi}[V^1](s) \leq \gamma ||V^1-V2||$. Thus, $||\cT^{\DS,\Xi}[V^2] - \cT^{\DS,\Xi}[V^1]||_\infty \leq \gamma ||V^1-V^2||$, indicating the contraction property of $\cT^{\DS,\Xi}[V]$. 

We now validate the \textbf{Markov optimality  (iii)}, i.e., $V^{\DS,\Xi}$ defined in \eqref{eq:DS-defineV} is the unique fixed point solution to the system $\cT^{\DS,\Xi}[V] = V$. 
The contraction property guarantees that the system $\cT^{\DS,\Xi}[V] = V$ always has a unique fixed point solution. We now prove that this a solution is also $V^{\DS,\Xi}$ defined  in \eqref{eq:DS-defineV}. Let $V^*$ be a solution to the above contraction system, then for any decision policy $\bPsi = \{\bpsi^0,\bpsi^1,\ldots\}$ and any sequence of distribution $\{\tilde{\bd}^0,\tilde{\bd}^1,\ldots\}$, where $\tilde{\bd}^t \in\Xi$ for $t=0,1,...$, we have
\begin{align}
V^*(s) &= \sup_{d_s \in \Xi_s} \left\{\bbE_{\beps_s \sim d_s}\left[\max_{a\in\cA}\left\{ r(a|s)+ \epsilon(a|s) + \gamma\sum_{s'}q(s'|s,a)V(s')\right\} \right] \right\}\nonumber \\
&\geq \bbE_{\beps_{s_0} \sim \tilde{d}^0_{s_0}}\left[ r(a_0|s_0)+ \epsilon(a_0|s_0) + \gamma\sum_{s'}q(s'|s_0,a_0)V(s')\Big|\: s_0=s, a_0 \sim \bpsi^0  \right]\nonumber\\
&\geq \bbE_{\tilde{\bd}^0,  \tilde{\bd}^1}\Bigg[ r(a_0|s_0)+ \epsilon(a_0|s_0) +  \gamma ( r(a_1|s_1)+ \epsilon(a_1|s_1)) + \gamma^2V^*(s_2) \nonumber\\
&\qquad \qquad\qquad \qquad \qquad\qquad\qquad \qquad\qquad\Big|\: s_0=s, (a_0,a_1) \sim (\bpsi^0,\bpsi^1)  \Bigg]\nonumber
\end{align}
By continuing expanding the recursion, we obtain
\begin{align}
V^*(s)
&\geq \bbE_{\tilde{\bd}^0,...,\tilde{\bd}^n}\Bigg[\sum_{t=0}^n \gamma^{[t]} (r(a_t|s_t)+ \epsilon(a_t|s_t)) + \gamma^{[n+1]}V^*(s_{n+1}) \nonumber\\
&\qquad \qquad\qquad \qquad \qquad\qquad\qquad \qquad\qquad\Big|\: s_0=s, (a_0,...,a_n) \sim (\bpsi^0,...,\bpsi^n)  \Bigg]\nonumber \\
&\geq \bbE_{\tilde{\bd}^0,...}\Bigg[\sum_{t=0}^\infty \gamma^{[t]} (r(a_t|s_t)+ \epsilon(a_t|s_t))\Bigg|\ s_0=s, (a_0,...)\sim \bPsi \Bigg] - \gamma^{[n+1]}||V^*||_\infty \nonumber\\
& - \gamma^{[n+1]}\bbE_{\tilde{\bd}^0,...}\Bigg[\sum_{t=n+1}^\infty \gamma^{[t-n-1]} (r(a_t|s_t)+ \epsilon(a_t|s_t))\Bigg|\ s_0=s, (a_0,...)\sim \bPsi \Bigg] \nonumber \\
& \geq \bbE_{\tilde{\bd}^0,...}\Bigg[\sum_{t=0}^\infty \gamma^{[t]} (r(a_t|s_t)+ \epsilon(a_t|s_t))\Bigg|\ s_0=s, (a_0,...)\sim \bPsi \Bigg] - \gamma^{[n+1]}||V^*||_\infty -  \frac{\gamma^{[n+1]}}{1-\gamma}\alpha\label{eq:DS-eq2}
\end{align}
where $\alpha = \sup_{s\in\cS}\sup_{d_s \in\Xi_s} \bbE\left[\max_a\Big\{r(a|s)+\epsilon(a|s)\Big\} \right]<\infty $ which is always finite, according to our assumptions. From \eqref{eq:DS-eq2}, by letting $n\rightarrow \infty$ to obtain,  for any decision policy $\bPsi = \{\bpsi^0,\bpsi^1,\ldots\}$ and sequence of distribution $\{\tilde{\bd}^0,\tilde{\bd}^1,\ldots\}$,
\begin{equation*}
V^*(s) \geq  \bbE_{\tilde{\bd}^0,...}\Bigg[\sum_{t=0}^\infty \gamma^{[t]} (r(a_t|s_t)+ \epsilon(a_t|s_t))\Bigg|\ s_0=s, (a_0,...)\sim \bPsi \Bigg]
\end{equation*}
Thus,
\begin{equation}
\label{eq:DS-eq3}
V^*(s) \geq  \sup_{\substack{\bd^t\in \Xi; \bpsi^t \\ t= 0,\ldots }}\bbE_{\substack{\beps^t \sim \bd^t \\ t=0,\ldots}}\Bigg[\sum_{t=0}^\infty \gamma^{[t]} (r(a_t|s_t)+ \epsilon(a_t|s_t))\Bigg|\ s_0=s, (a_0,...)\sim \bPsi \Bigg]
\end{equation}
Now, to prove the opposite direction of Inequality \eqref{eq:DS-eq3}, the Bellman equation tells use  that for any $\xi>0$, there are  distributions $\{\tilde{\bd}^0,\tilde{\bd}^1,...\}$ and decision rule $\bPsi = \{\bpsi^0,\bpsi^1,\ldots\}$ such that 
\begin{align}
    V^*(s) &\leq \bbE_{\beps^0_{s_0} \sim \tilde{d}^0_{s_0}}\left[ r(a_0|s_0)+ \epsilon(a_0|s_0) + \gamma V^*(s_1)\Big|\: s_0=s, a_0 \sim \bpsi^0  \right] + \xi \nonumber\\
    &\leq \bbE_{\tilde{\bd}^0,...,\tilde{\bd}^n}\Bigg[\sum_{t=0}^n \gamma^{[t]} (r(a_t|s_t)+ \epsilon(a_t|s_t)) + \gamma^{[n+1]}V^*(s_{n+1}) \nonumber\\
    &  \qquad\qquad\Big|\: s_0=s, (a_0,...,a_n) \sim (\bpsi^0,...,\bpsi^n)  \Bigg] + \xi\sum_{t=0}^n \gamma^{[t]}\nonumber \\
    &\leq \bbE_{\tilde{\bd}^0,...}\Bigg[\sum_{t=0}^\infty \gamma^{[t]} (r(a_t|s_t)+ \epsilon(a_t|s_t))\Bigg|\ s_0=s, (a_0,...)\sim \bPsi \Bigg] + \gamma^{[n+1]}||V^*||_\infty \nonumber\\
& + \gamma^{[n+1]}\bbE_{\tilde{\bd}^0,...}\Bigg[\sum_{t=n+1}^\infty \gamma^{t-n-1} (r(a_t|s_t)+ \epsilon(a_t|s_t))\Bigg|\ s_0=s, (a_0,...)\sim \bPsi \Bigg]+ \xi\sum_{t=0}^n \gamma^{[t]} \nonumber \\
&\leq \bbE_{\tilde{\bd}^0,...}\Bigg[\sum_{t=0}^\infty \gamma^{[t]} (r(a_t|s_t)+ \epsilon(a_t|s_t))\Bigg|\ s_0=s, (a_0,...)\sim \bPsi \Bigg] + \gamma^{[n+1]}||V^*||_\infty +  \nonumber \\
&\qquad\qquad + \frac{\gamma^{[n+1]}\alpha}{1-\gamma} + \xi\sum_{t=0}^n \gamma^{[t]} \nonumber
\end{align}
So let $n \rightarrow \infty$  we have
\[
V^*(s)\leq  \sup_{\substack{\bd^t\in \Xi; \bpsi^t\\ t=0,\ldots}}\bbE_{\substack{\beps^t \sim \bd^t \\ t=0,\ldots}}\Bigg[\sum_{t=0}^\infty \gamma^{[t]} (r(a_t|s_t)+ \epsilon(a_t|s_t))\Bigg|\ s_0=s, (a_0,...)\sim \bPsi \Bigg] +\frac{\xi}{1-\gamma}, 
\]
then by letting $\xi\rightarrow 0$ and combining with \eqref{eq:DS-eq3} we have
\[
V^*(s) =   \sup_{\substack{\bd^t\in \Xi; \bpsi^t\\ t=0,\ldots}}\bbE_{\substack{\beps^t \sim \bd^t \\ t=0,\ldots}}\Bigg[\sum_{t=0}^\infty \gamma^{[t]} (r(a_t|s_t)+ \epsilon(a_t|s_t))\Bigg|\ s_0=s, (a_0,...)\sim \Psi \Bigg] = V^{\DS,\Xi}(s),\;\forall s\in \cS,
\]
which is the desired result. 

\textbf{For (iv)}, from the above results, we see that an optimal decision rule can be computed as 
\[
\psi^*(s|\beps_s) = \text{argmax}_{a} \left\{r(a|s) + \epsilon(a|s) + \gamma \sum_{s'}q(s'|a,s) V^{\DS,\Xi}(s')\right\}
\]
and assume that $d^*_s$ is the probability distribution of $\beps_s$ that attains the maximization in the Bellman update $\cT^{\DS,\Xi}[V(s)$, i.e., $$d^*_s= \text{argmax}_{d_s\in\Xi_s} \bbE_{\beps_s \sim d_s}\left[\max_{a} \Big\{w_{sa} + \epsilon(a|s) \Big\}\right],$$
where $w_{sa} = r(a|s) + \gamma \sum_{s'}q(s'|a,s) V^{\DS,\Xi}(s')$ for any $a\in\cA,s\in\cS$.
The density of the optimal decision rule $\bpsi_s$ (i.e. optimal policy) becomes 
\begin{align}
\pi^{\DS,*}(a|s) &= \bbP\left[\psi^*(s|\beps_s) = a  \Big| \beps \sim d^*_s\right] \nonumber \\
&\stackrel{(i)}{=} \frac{\partial \varphi_s(\bw_s)}{\partial w_{sa}},\nonumber
\end{align}
where $\varphi_s(\bw_s) = \bbE_{\beps_s \sim d^*_s}[\max_a \{ w_{sa}+\epsilon(a|s)\}]$ and (i) is due to \cite{Mcfadden1981}. Note that
\[
 \varphi_s(\bw_s) = \sup_{d_s\in\Xi_s}\left\{\bbE_{\beps_s \sim d_s}[\max_a \{ w_{sa}+\epsilon(a|s)\}]\right\}.
\]
We obtain the desired results and complete the proof.

%-------------------------------
%
%-------------------------------

\subsection{Proof of Proposition \ref{th:Ev-ER-MDPs}}
Consider an ER-MDPs defined by the following maximization of expected discounted regularized reward
\begin{align}
\sup_{\bpii^0,\ldots }\Bigg\{ \bbE_{\tau \sim (\bPi,\bq)}\Bigg[\sum_{t=0}^{\infty}\gamma^{[t]} \Big(r(a_t|s_t) - \eta \ln(\pi^t(a_t|s_t)) \Big)  \Bigg]\Bigg\},
\end{align}
where $\eta \geq 0$. According to \cite{Ziebart2008maximum,Haarnoja2017RL_ER}, the following policy  is optimal to the ER-MDP problem under reward function $\br$
\begin{equation}
\label{eq:proof-ER-EV-eq1}
\pi^{\EV,\eta}(a|s) = \frac{\exp\Big(h(a|s)/\eta\Big)}{\sum_{a'\in\cA} \exp\Big(h(a'|s)/\eta\Big)}, \ \forall s\in\cS, a\in\cA
\end{equation}
where $h(a|s) = r(a|s) +  \gamma \sum_{s'}q(s'|s,a) V^{\EV}(s')$,  $V^{\EV}$ is the unique fixed point solution to the system of equations 
\[
V(s) = \eta \ln \Bigg(\sum_{a\in\cA} \exp\Big(   \frac{1}{\eta}r(a|s) +  \gamma \frac{1}{\eta} \sum_{s'}q(s'|s,a) V(s') \Big) \Bigg).
\]
Now look at the an EV-MDP defined as follows
\begin{align}
   \sup_{\bpsi^0,\bpsi^1,...}  \bbE_{\beps} \left\{ \left[  \bbE_{\tau \sim (\bPsi,\bq)}\Bigg[\sum_{t=0}^{\infty}\gamma^{[t]} \Big(r(a_t|s_t) + \eta\epsilon(a_t|s_t)\Big) \Bigg] \right]\right\},\label{eq:MDP-SR}
\end{align}
where $\epsilon(a|s)$, $\forall a,s$, are i.i.d and  follow the Gumbel distribution, then \cite{Rust1987GMC} show that the MDP problem \eqref{eq:MDP-SR} yield a value function  and  a set of optimal choice probabilities (i.e. policy) that are exactly equal to those given in  \eqref{eq:proof-ER-EV-eq1}. This indicates the equivalence between EV-MDPs and ER-MDPs. We also note here that if $\eta = 0$, which describes the case that the entropy terms $\pi(a|s)\ln \pi(a|s)$ or the noise term $\epsilon(a|s)$ are neglected from the MDPs, then the EV-MDP and ER-MDP become the standard unconstrained MDP. 
On the other hand, it is not difficult to give an example of MDP such that any optimal policy to the standard MDP need to satisfy $\pi^*(a|s) $ only takes a value of 0 or 1, but any optimal policy to the ER-MDP problem satisfies $\pi^{*,\EV}(a|s) \in (0,1)$ for some pairs of $(a,s)\in\cA\times\cS$. This validates the claim that ER-MDPs and ER-MDPs strictly subsumes the standard MDPs. We complete the proof.

%-------------------------------
%
%-------------------------------

\subsection{Proof of Theorem \ref{th:R-and-DS-MDP}}
To prove the equivalence, let us  consider a R-MDP model specified by a tuple $(\cS,\cA,\bpii,\bq,\br,\gamma)$ and by the set of concave function $\bphi = \{\phi_s,\ s\in\cS\}$,  and a DS-MDP model specified by the same tuple and  ambiguity sets $\{\Xi_s, s\in\cS \}$. Theorem \ref{th:DS-contractionmapping-optimal-policy} and \cite{Geist2019theory} tell us that these two models yield the following value functions that are unique solutions to the following contraction systems
\begin{align}
    \cT^{\RE,\bphi}[V](s) &= \sup_{\bpii_s} \big\{\bw_s^\T \bpii_s + \phi_s(\bpii_s)\big\} \label{eq:RE-contraction-eq1} \\
        \cT^{\DS,\Xi}[V](s) &= \sup_{d_s \in \Xi_s}  \Big\{\bbE_{\beps_s \sim d_s}[\max_a \{ w_{sa}+ \epsilon(a|s)\}] \Big\}, \label{eq:DS-contraction-eq1} 
\end{align}
where $\bw_s$ is a vector of size $|\cA|$ with entries $w_{sa} = r(a|s)+ \gamma \sum_{s'}q(s'|s,a)V(s')$. 
Using Theorem 1 in   \cite{Feng2017relationDC}, we have that for any strictly convex function $\widetilde{\phi}_s(\cdot)$, there is always a distribution set $\widetilde{\Xi}_s$ such that $\cT^{\RE,\widetilde{\bphi}}[V](s) = \cT^{\DS,\widetilde{\Xi}}[V](s)$ for any $\bw_s\in \bbR^{|\cA|}$, and vise-versa. As a result, for any set of convex functions $\widetilde{\bphi} = \{\widetilde{\phi}_s,s\in\cS\}$ there always a set $\widetilde{\Xi}  = \{\widetilde{\Xi}_s,s\in\cS\}$ such that  $\cT^{\RE,\bphi}[V] = \cT^{\DS,\Xi}[V]$ for any $V \in \bbR^{|\cS|}$, and vice-versa. So, for any R-MDP model, there is a DS-MDP model such that the two contraction systems yield the same value functions, i.e., $V^{\RE,\widetilde{\bphi}} = V^{\DS,\widetilde{\Xi}}$, and vice-versa. On  other hand, the optimal policies given by the two MDP models are
\begin{align}
    \bpii^{\RE,\widetilde{\bphi}}_s &= \text{argmax}_{\bpii_s}\Big\{ (\bw^{\RE}_s)^\T \bpii_s  + \widetilde{\phi}_s(\bpii_s);\Big\} \\
    \bpii^{\DS,\widetilde{\Xi}}_s &= \nabla_{\bw} \Bigg\{\sup_{d_s \in \widetilde{\Xi}_s}  \Big\{\bbE_{\beps_s \sim d_s}[\max_a \{ w^{\DS}_{sa}+ \epsilon(a|s)\}] \Big\}\Bigg\}
\end{align}
where $\bw^{\RE}_s$ and $\bw^{\DS}_s$ are vectors of size $|\cA|$ with entries $w^{\RE}_{sa} =  r(a|s)+ \gamma\sum_{s'}q(s'|s,a)V^{\RE,\widetilde{\bphi}}(s')$ and $w^{\RE}_{sa} =  r(a|s)+ \gamma\sum_{s'}q(s'|s,a)V^{\DS,\widetilde{\bphi}}(s')$ for all $a\in\cA$. Since  $V^{\RE,\widetilde{\bphi}} = V^{\DS,\widetilde{\Xi}}$, $\bw^{\RE}_s = \bw^{\DS}_s$ for all $s\in\cS$ and  any reward function $\br$. 
Moreover, from the way we select $\widetilde{\phi}_s$ and $\widetilde{\Xi}_s$ we have
\[
\varphi_s(\bw_s) = \max_{\bpii_s}\Big\{ (\bw_s)^\T \bpii_s  + \widetilde{\phi}_s(\bpii_s)\Big\} = \sup_{d_s \in \widetilde{\Xi}_s}  \Big\{\bbE_{\beps_s \sim d_s}[\max_a \{ w_{sa}+ \epsilon(a|s)\}] \Big\},\ \forall \bw_s\in\bbR^{|\cA|}.
\]
From \cite{Feng2017relationDC} we also have  
\[
\nabla_{\bw_s} \varphi_s(\bw_s) = \text{argmax}_{\bpii_s}\Big\{ (\bw_s)^\T \bpii_s  + \widetilde{\phi}_s(\bpii_s)\Big\}.
\]
Thus,  $ \bpii^{\RE,\widetilde{\bphi}}_s =  \bpii^{\DS,\widetilde{\Xi}}_s $ for any reward function $\br$. So, in summary, for any R-MDP model specified by a set $\widetilde{\bphi}$, there is always a DS-MDP model specified by a distributions set $\widetilde{\Xi}$  such that $\bpii^{\RE,\widetilde{\bphi}}_s =  \bpii^{\DS,\widetilde{\Xi}}_s $ for any $s$, and vice-versa. In other words, the R-MDP and DS-MDP frameworks are equivalent. 

To prove the claim that the R-MDP framework strictly subsumes the S-MDP one, we just use a result from \cite{Feng2017relationDC} (for one-step decision problems) saying that if $|\cA|\geq 3$ there exist a convex function $\phi_s:\bbR^{|\cA|}\rightarrow \bbR$ such that there is no distribution $d_s$ such that the following holds for any $\bw_s \in \bbR^{|\cA|}$
\[
 \text{argmax}_{\bpii_s \in\Delta(|\cA|)}\Big\{ (\bw_s^\T \bpii_s  + {\phi}_s(\bpii_s)\Big\} = \nabla_{\bw} \Bigg\{  \Big\{\bbE_{\beps_s \sim d_s}[\max_a \{ w_{sa}+ \epsilon(a|s)\}] \Big\}\Bigg\}.
\]
We now complete the proof.

%-------------------------------
%
%-------------------------------

\subsection{Proof of Proposition  \ref{prop:S-MDP-subsume-EV-MDP}}

It is clear that the S-MDPs covers the ER-MDPs, thus also covers the EV-MDP framework. To show that there is a S-MDP model that is not equivalent to any EV-MDP model, we just take a small MDP example of three states $\cS = \{s_0,s_1,s_2\}$ and assume that there are two possible action $\cA = \{a_1,a_2\}$. The transition probabilities are given as $q(s_1|a_1,s_0) = 1$, $q(s_2|a_2,s_0) = 1$, $q(s_1|a_1,s_1) =  q(s_1|a_2,s_1) = 1$ and $q(s_2|a_2,s_2) = q(s_2|a_1,s_2)$. In other words, we can move from state $s_0$ to $s_1$  with probability 1 by making action $a_1$ and move from $s_0$ to $s_2$ by making action $a_2$. From states $s_1$ or $s_2$, we have a probability of 1 to stay at the same state, under any action. We further assume that $r(a|S_1) = r(a|s_2) = 0$ for all $a\in \cA$. It is easy to see that an optimal policy under an EV-MDP or ER-MDP model with parameter $\eta>0$ is given as
\begin{equation}
\label{eq:S-subsum-EV-eq0}
\pi^{\EV,\eta}(a_1|s_0) = \frac{e^{r(a_1|s_0)/\eta}}{e^{r(a_1|s_0)/\eta}+e^{r(a_2|s_0)/\eta}};\ \pi^{\EV,\eta}(a_2|s_0) = \frac{e^{r(a_2|s_0)/\eta}}{e^{r(a_1|s_0)/\eta}+e^{r(a_2|s_0)/\eta}};  
\end{equation}
which leads to 
\begin{equation}
\label{eq:S-subsum-EV-eq1}
\frac{\pi^{\EV,\eta}(a_1|s_0)}{\pi^{\EV,\eta}(a_2|s_0)} = \exp\left(r(a_1|s_0)/\eta - r(a_2|s_0)/\eta\right).    
\end{equation}
Now, to show that the S-MDP framework strictly subsumes the EV-MDP one, we consider the above example and use a S-MDP model with $\epsilon(a|s) = 0$ for all $(a,s)\in \cA\times\cS$ except $\epsilon(a_1|s_0)$ that follows a uniform distribution in  $[0,1]$. Assume that the S-MDP with reward $\br'  =\br+ \bap$ will yield similar optimal policies for any $\br\in\bbR^{|\cS|\times|\cA|}$.
Clearly, the S-MDP model has the following optimal policy
\[
\pi^{\SR,\beps}(a_1|s_0) = \bbP\big[ r(a_1|s_0) + \alpha_{s_0a_1}+ \epsilon(a_1|s_0) \geq r(a_2|s_0) + \alpha_{s_0a_2}\big]
\]
and the ration between the two action properties under the S-MDP model becomes
\begin{align}
\frac{\pi^{\SR,\beps}(a_1|s_0)}{\pi^{\SR,\beps}(a_2|s_0)}
 &\stackrel{(i)}{=}  \frac{\bbP\big[\epsilon(a_1|s_0) \geq r(a_2|s_0) - r(a_1|s_0) + \beta_{s_0} \big] }{\bbP\big[ \epsilon(a_1|s_0) < r(a_2|s_0) - r(a_1|s_0) +\beta_{s_0}\big]} \nonumber \\
& =\begin{cases}
0 & \text{ if }r(a_2|s_0) - r(a_1|s_0) + \beta_{s_0} \leq 0\\
\infty & \text{ if }r(a_2|s_0) - r(a_1|s_0) + \beta_{s_0} \geq 1 \\
\frac{r(a_2|s_0) - r(a_1|s_0) + \beta_{s_0}}{1-(r(a_2|s_0) - r(a_1|s_0) + \beta_{s_0})} &\text{ if }r(a_2|s_0) - r(a_1|s_0) + \beta_{s_0} \in (0,1),
\end{cases}
\label{eq:S-subsum-EV-eq2}    
\end{align}
where in \textit{(i)}, $\beta_{s_0} = \alpha_{s_0a_2} - \alpha_{s_0a_1}$. Now we note that  the probabilities given in \eqref{eq:S-subsum-EV-eq0} is a unique solution to the EV-MDP model. Moreover, it is clear that the two ratios in \eqref{eq:proof-ER-EV-eq1} and \eqref{eq:S-subsum-EV-eq2} can not be equal for any reward function $\br$. This implies that the EV-MDP framework does not cover the above S-MDP model. Thus, the S-MDP framework strictly subsumes the ER-MDP and ER-MDP ones. We complete the proof.  

%-------------------------------
%
%-------------------------------

\subsection{Proof of Theorem \ref{th:CT-MDP-and-others}}
We first note that, in a CT-MDP model, the value function is defined as 
\[
V^{\CC,\cD^\pi}(s) = \sup_{\bpii \in \cD^\pi}\Bigg\{ \bbE_{\tau \sim (\bpii,\bq)}\Bigg[\sum_{t=0}^{\infty}\gamma^{[t]} r(a_t|s_t)\ \Bigg| s=s_0  \Bigg]\Bigg\},
\]
and the Bellman equation
\begin{equation}
\label{eq:bellman-CT-MDP}
\cT^{\CC,\cD^\pi}[V](s) = \sup_{\bpii_s\in\cD^\pi_s} \left\{\bbE_{\bpii_s} \left[r(a|s) + \gamma \sum_{s'}q(s'|s,a) V(s')\right]\right\}.
\end{equation}
Then we see that $\cT^{\CC,\cD^\pi}[V]$ is a contraction of parameter $\gamma$ and  $V^{\CT,\cD^\pi}$ is the unique solution to the system $\cT^{\CC,\cD^\pi}[V] = V$. 
Here the contraction property holds for any nonempty set $\cD^\pi_s$, not necessarily convex or compact (see Section \ref{sup:sec:contraction-CT-MDP} below for a  more detailed discussion about this). 
To compute an optimal policy, we need to solve the Bellman equation \eqref{eq:bellman-CT-MDP} to obtain $V^{\CC,\cD^\pi}$ and solve the following problem to get an optimal policy
\[
\bpii^*_s = \text{argmax}_{\bpii_s\in\cD^\pi_s} \left\{\bbE_{\bpii_s} \left[r(a|s) + \gamma \sum_{s'}q(s'|s,a) V^{\CT,\cD^\pi}(s')\right]\right\}.
\]
Now, we prove the results by contradiction. 
For the sake of simplicity, we just use a small MDPs of three states $\cS = \{s_0,s_1,s_2\}$ and two actions $\cA = \{a_1,a_2\}$ as  in the proof of Proposition \ref{prop:S-MDP-subsume-EV-MDP}. Now, consider an ER-MDP model based on this set of states and actions, we see that the optimal policy for a given reward function $\br$ becomes $\pi(a_1|s_0) = \exp(r(a_1|s_0))/R$ and $\pi(a_2|s_0) = \exp(r(a_2|s_0))/R$, where $R= \exp(r(a_1|s_0))+ \exp(r(a_2|s_0))$. Clearly, $\bpii_{s_0}$ will span all over the simple $\Delta(2)$ when the reward $\br$ varies in $\bbR^2$. So if there is a CT-MDP model that is equivalent to the above ER-MDP model, than $\cD^{\pi}_{s_0}$ has to cover all the simplex $\Delta(2)$ as well, i.e.,  the  CT-MDP model becomes a standard unconstrained one, thus yielding deterministic optimal policies (i.e. the probability of making an action only takes a value of 0 or 1).  

For the second statement, we consider a CT-MDP model with a singleton feasible set that contains only one deterministic solution $\widetilde{\bpii}_{s_0} = [1,0]$. Now, assume that there is a R-MDP model specified by a set of concave function $\bphi$ 
that is equivalent to the above CT-MDP model. That is, for any reward function $\br$, the R-MDP model always yields an optimal policy such that $\pi(a_1|s_0) = 1$ and $\pi(a_1|s_0) = 0$. We will prove that this cannot be true for any reward function $\br$. To this end, we look at the Bellman equation at state $s_0$,
\begin{equation}
\label{eq:CT-MDP-eq1}
\bpii^*_{s_0} = \text{argmax}_{\bpii_{s_0}}\Big\{\sum_{a\in\cA} r(a|s_0)\pi(a|s_0) + \phi(\bpii_{s_0})\Big\},
\end{equation}
with a note that $\phi_{s_0}(\bpii_{s_0})$ is bounded in $\Delta(2)$, i.e., there exists $U^{\phi}, L^{\phi} \in \bbR$ such that $L^\phi \leq \phi(\bpii_{s_0}) \leq U^\phi$, $\forall \bpii_{s_{0}} \in\Delta(2)$. So if we set $r(a_1|s_0) = 0$ and $r(a_1|s_0) = |U^{\phi}- L^{\phi}|+1$, then we have $\phi([1,0]) -\phi([0,1])  < U^{\phi}-L^{\phi} < r(a_2|s_0)$, or $r(a_1|s_0)+ \phi_{s_0}([1,0]) < r(a_2|s_0)+ \phi_{s_0}([0,1])$, implying that $\widetilde{\bpii}_{s_0} = [1,0]$ is not an optimal solution to  \eqref{eq:CT-MDP-eq1}. This contradicts our initial assumption on the equivalence between the R-MDP and CT-MDP models.  
We complete the proof.

\subsection{Proof of Proposition  \ref{prop:ER-to-CT}}

For any vector $\bw_s\in \bbR^{|\cA|}$, let $\bpii^{*}_s$ be an optimal solution to the problem $\sup_{\bpii_s}\{\bw^\T_s \bpii_s +\phi_s(\bpii_s)\}$. Moreover, we choose
%, where $\bw_s$ is a vector in   $\bbR^{|\cA|}$ with elements  $\bw_{sa} = r(a|s)+\sum_{s'}q(s'|s,a)V^{\RE,\bphi}(s')$, where $V^{\RE,\bphi}$ is the value function of the R-MDP model. 
  $c_s = -\phi_s(\bpii^{*}_s)$ and  consider the two following problems
\begin{equation*}
\text{(P1)}\qquad    \sup_{\bpii_s \in \Delta(|\cA|)}\left\{  \bw^\T_s \bpii_s + \phi_s(\bpii_s)\right\}
\end{equation*}
\begin{equation*}
    \text{(P2)}\qquad    \sup_{\bpii_s \in \Delta(|\cA|)}\left\{  \bw^\T_s \bpii_s - c_s \Big|\  -\phi_s(\bpii_s)\leq c_s\right\}
\end{equation*}
We will show that any optimal solution to (P2) is also optimal to (P1). 
Clearly, if $\bpii^*_s$ is optimal to (P1) than it is feasible to (P2).
Now let $\bpii^{**}_s$ is optimal to (P2), from the feasibility of $\bpii^{*}_s$ to (P2) we have
\begin{equation}
\label{eq:CT-MDP-eq2}
\bw_s^\T \bpii^*_s \leq \bw_s^\T \bpii^{**}_s,    
\end{equation}
which leads to 
\[
\bw_s^\T \bpii^*_s + \phi_s(\bpii^*_s) \stackrel{(i)}{=} \bw_s^\T \bpii^*_s -c_s \stackrel{(ii)}{\leq} \bw_s^\T \bpii^{**}_s + \phi_s(\bpii_s^{**}),
\]
where (i) is because $c_s = -\phi_s(\bpii^{\RE,\bphi}_s) = -\phi_s(\bpii^*_s)$ and  (ii) is due to \eqref{eq:CT-MDP-eq2} and  the fact that $\bpii^{**}_s$ is feasible to (P2), thus $\phi_s(\bpii^{**}_s)\geq -c_s$.
Since $\bpii^*_s$ is optimal to (P1), $\bpii^{**}_s$ is also optimal to (P1). 
Moreover, (P1) has a strictly concave objective function, it always yields a unique optimal solution $\bpii^{*}_s$. Thus for any $\bpii^{**}_s$ optimal to (P2), we have $\bpii^{**}_s = \bpii^{*}_s$.

Now, let  $\bpii^*$ is an optimal  solution to a R-MDP problem with $\bphi  = \{\phi_s,s\in\cS\}$ and define $c_s = -\phi_s(\bpii^*_s)$ and $\cD^\pi_s = \{\bpii_s\in\Delta(|\cA|)| -\phi_s(\bpii_s)\leq c_s\}$. From the above results, we see that the two following systems always yield the same optimal solution and optimal value.
\begin{align}
    &\sup_{\bpii_s} \left\{\bbE_{\bpii_s}\left[ r(a|s) + \gamma \sum_{s'}q(s'|a,s)V^{\RE,\bphi}(s') \right]+ \phi_s(\bpii_s)\right\}\nonumber \\
    &\max_{\bpii_s\in\cD^\pi_s}\left\{\bbE_{\bpii_s}\left[ r(a|s) + \gamma \sum_{s'}q(s'|a,s)V^{\RE,\bphi}(s') \right]\right\}. \nonumber
\end{align}
Therefore, $\cT^{\RE,\bphi}[V^{\RE,\bphi}](s) = \cT^{\CT,\cD^{\pi}}[V^{\RE,\bphi}](s)$. Since $V^{\RE,\bphi}$ is a solution to the system $\cT^{\RE,\bphi}[V] = V$, we have
\[
\cT^{\CT,\cD^{\pi}}[V^{\RE,\bphi}](s) = V^{\RE,\bphi}(s).
\]
Hence, $V^{\RE,\bphi}$ is also a fixed point solution to $\cT^{\CT,\cD^{\pi}}[V] = V$, i.e., $V^{\RE,\bphi} = V^{\CT,\cD^\pi}$. 
As a result, the two models  yield the same optimal policy. We complete the proof.  

%-------------------------------
%-------------------------------
%-------------------------------

\section{Relevant Discussions}

\label{sup:sec:Examples-DS-MDP}

\subsection{Contraction Property of $\cT^{\CT,\cD^\pi}$}
\label{sup:sec:contraction-CT-MDP}

We revisit the contraction property  of $\cT^{\CT,\cD^\pi}[V]$ under arbitrarily non-empty feasible sets $\{\cD^\pi_s,\;  s\in\cS\}$. The proof  is similar to the unconstrained model but we provide it here for the sake of self-contained. For any $V^1,V^2\in\bbR^{|\cA|}$, for each state $s\in\cS$,  we consider two cases: $\cT^{\CC,\cD^\pi}[V^1](s) \geq \cT^{\CC,\cD^\pi}[V^2](s)$ or    $\cT^{\CC,\cD^\pi}[V^1](s) <\cT^{\CC,\cD^\pi}[V^2](s)$. The following evaluations are for the formal case and the latter can be done analogously.
For any $\xi>0$, let $\overline{\bpii}^1_s \in\cD^\pi_s$ a solution such that
\[
\sup_{\bpii_s\in\cD^\pi_s} \left\{ \bbE_{\bpii_s} \left[ r(a|s) + \gamma \sum_{s'}q(s'|s,a) V^1(s')\right]\right\} \leq \bbE_{\overline{\bpii}^1_s} \left[ r(a|s) + \gamma \sum_{s'}q(s'|s,a) V^1(s')\right] +\xi,
\]
and we write, for any $s\in\cS$,
\begin{align}
    &|\cT^{\CC,\cD^\pi}[V^1](s) - \cT^{\CC,\cD^\pi}[V^2](s)|  = \cT^{\CC,\cD^\pi}[V^1](s) - \cT^{\CC,\cD^\pi}[V^2](s)\nonumber\\
    &\leq \bbE_{\overline{\bpii}^1_s} \left[ r(a|s) + \gamma \sum_{s'}q(s'|s,a) V^1(s')\right] - \bbE_{\overline{\bpii}^1_s} \left[ r(a|s) + \gamma \sum_{s'}q(s'|s,a) V^2(s')\right] + \xi\nonumber\\
    &= \gamma \sum_{a\in\cA,s'\in\cS}\overline{\pi}^1(a|s)q(s'|s,a) (V^1(s')-V^2(s'))+\xi\nonumber\\
    &\leq \gamma ||V^1-V^2||_\infty+\xi.\nonumber
\end{align}
Thus, $||\cT^{\CC,\cD^\pi}[V^1] - \cT^{\CC,\cD^\pi}[V^2]||_\infty  \leq \leq \gamma ||V^1-V^2||_\infty + \xi$. Since  $\xi$ can be arbitrarily small, we let $\xi\rightarrow 0$ to have  $||\cT^{\CC,\cD^\pi}[V^1] - \cT^{\CC,\cD^\pi}[V^2]||_\infty  \leq \leq \gamma ||V^1-V^2||_\infty$. The case that $\cT^{\CC,\cD^\pi}[V^1](s) <\cT^{\CC,\cD^\pi}[V^2](s)$ can be done similarly, which validates the contraction property of $\cT^{\CC,\cD^\pi}[V]$. It is also clear that the contraction property holds for any non-empty set $\cD^\pi$.

Now we validate the result that if $V^*$ is the unique solution of the contraction mapping $\cT^{\CC,\cD^\pi}[V] = V$, then 
\[
V^*(s) = V^{\CT,\cD^{\pi}}(s) = \sup_{\bpii^t \in \cD^\pi; t=0,\ldots}\Bigg\{ \bbE_{\tau \sim (\bPi,\bq)}\Bigg[\sum_{t=0}^{\infty}\gamma^{[t]} r(a_t|s_t)\ \Bigg| s=s_0  \Bigg]\Bigg\}.
\]
This can be done similarly as in the unconstrained MDP case, thus we skip the rest. 

%-------------------------------
%
%-------------------------------

\subsection{From CT-MDPs to R-MDPs}
\label{sup:sec:from-R-MDP-to-CT-MDP}
The following proposition shows that a CT-MDP model specified by a set of convex constraints can be equivalently represented by a R-MDP model. The proposition is easy to validate, using Lagrange duality.
\begin{proposition}
\label{prop:CT-to-RE}
 Given any CT-MDP model $\bX$ defined by a tuple $(\cS,\cA,\bq, T=\infty,\br,\gamma)$ and  feasible sets $\cD^{\pi}_s = \{\bpii_s\in \Delta(|\cA|)|\ l^k_s(\bpii_s) \leq c^k_s, k=1,...,K_s\}$, where $l^k_s(\bpii_s)$ are strictly convex in $\Delta(|\cA|)$. Assuming also that the Slater conditions hold for all $s\in\cS$, i.e., $c^k_s> \inf_{\bpii_s} l^k_s(\bpii_s),\, \forall k\in[K], s\in\cS$, then there are a set of constants $0\leq \overline{\lambda}_{sk}<\infty$, $\forall s,k$ such that the R-MDP model defined by the same tuple with the set of regularizers $\{\phi_s(\bpii_s) = -\sum_{k}\overline{\lambda}_{sk} (l^k_s(\bpii_s)-c_s),\ \forall s \in \cS\}$  yields an identical optimal policy. 
\end{proposition}
\begin{proof}
We write the Bellman update of the CT-MDP model as follows
\begin{align}
	 g_s(\bw_s) = \underset{\bpii_s}{\text{maximize}}\qquad & \bw^\T_s \bpii_s &\label{eq:R-CT-eq1}  \\
	 \text{subject to} \qquad & l^k_s(\bpii_s) \leq c^k_s,\; k = 1,\ldots,K_s  \nonumber\\
	  & \bpii_s \in \Delta(|\cA|).&\nonumber
\end{align}
The \textit{Lagrange dual function} is defined  as
\[
\cL_s (\bld|\bw_s) =  \max_{\bpii_s \in 
\Delta(|\cA|)} \left\{ \bw^\T_s \bpii_s - \sum_{k=1}^{K_s}\lambda_k (l^k_s(\bpii_s) - c^k_s)\right\}.
\]
Since the \textit{Slater conditions} hold, we always have
\[
g_s(\bw_s) = \min_{\bld\geq 0} \left\{\cL_s(\bld|\bw_s)\right\}.
\]
Now, let $\overline{\bld}_s$ be an optimal solution to $\left\{\cL_s(\bld|\bw_s)\right\}$. The KKT conditions imply that 
\begin{align}
    g_s(\bw_s) &= \cL_s(\overline{\bld}_s|\bw_s) \nonumber \\
    \overline{\lambda}_{sk} (l^k_s(\bpii_s) - c^k_s) &= 0,\ \forall k. \nonumber
\end{align}
So, if we define $\phi_s(\bpii_s) = -\sum_{k=1}^{K_s}\overline{\lambda}_{sk} (l^k_s(\bpii_s) - c_s)$, then the problem 
$\max_{\bpii_s}\left\{\bw^\T_s \bpii_s +\phi_s(\bpii_s)\right\} $ will yield the same optimal value and optimal solutions as \eqref{eq:R-CT-eq1}.   Then, in analogy to the proof of Proposition \ref{prop:ER-to-CT}, we can also show that the R-MDP model specified by $\bphi = \{\phi_s,s \in\cS\}$  and the CT-MDP will yield the same optimal value and optimal policy.  
\end{proof}

\subsection{Constrained MDPs and Some Computational Complexity Remarks}
In this section, we look at some CT-MDP models and investigate their computational complexities. An mentioned, constraints based on a KL divergence would be useful to control the distance between a desired policy and a given policy. The Bellman update can be written as 
\begin{align}
	 \cT^{\CT,\cD^\pi}[V](s) = \underset{\bpii_s}{\text{maximize}}\qquad & \bw^\T_s \bpii_s &\label{prop:CT-SM-eq1}  \\
	 \text{subject to} \qquad & \KL (\bpii_s|\overline{\bpii}_s) \leq c_s  \nonumber\\
	  & \bpii_s \in \Delta(|\cA|),&\nonumber
\end{align}
where $\bw_s \in\bbR^{|\cA|}$ with entries $w_{sa} = r(a|s) + \gamma \sum_{s'}q(s'|a,s)V(s')$, for all $a\in\cA$. 
The following proposition shows how to solve \eqref{prop:CT-SM-eq1} efficiently (in polynomial time).
\begin{proposition}
Problem \eqref{prop:CT-SM-eq1} is equivalent to
\[
-\min_{y\in\bbR_+ }\left\{c_s y + y \log \left(\sum_{a}\overline{\bpii}(a|s)\exp\left(\frac{w_{sa}}{y}\right)\right)\right\},
\]
which can be solved by bisection and the complexity to get an $\epsilon$-optimal solution is $\cO(|\cA|\ln (1/\epsilon))$.
\end{proposition}
\begin{proof}
We convert the problem into a minimization problem and the proof is similar to Lemma 4.1 of \cite{Iyengar2005robustDP}. 
\end{proof}

If the distance is restricted by a L1 norm, i.e., $\cD^\pi_s = \{\bpii_s|\ ||\bpii_s-\overline{\bpii}_s||_1 \leq c_s\}$, we have the following Bellman update
\begin{align}
	 \cT^{\CT,\cD^\pi}[V](s) = \underset{\bpii_s}{\text{maximize}}\qquad & \bw^\T_s \bpii_s &\label{prop:CT-SM-eq2}  \\
	 \text{subject to} \qquad & ||\bpii_s - \overline{\bpii}_s)||_1 \leq c_s,  \nonumber\\
	  & \bpii_s \in \Delta(|\cA|)&\nonumber
\end{align}
We also have the proposition.
\begin{proposition}
The objective value of \eqref{prop:CT-SM-eq2} is equal to
\[
\bw^\T_s \overline{\bpii}_s + \frac{c_s}{2} \min_{\bmu_s\in \bbR^{|\cA|}_+}\left\{\max_a \{w_{sa}+\mu_{sa}\} - \min_{a}\{w_{sa}+\mu_{sa}\} \right\},
\]
which can be solved \textit{exactly} in $\cO(|\cA|\ln |\cA|)$.
\end{proposition}
\begin{proof}
Similar to the proof of Lemma 4.3 in \cite{Iyengar2005robustDP}. Note that in this case we can solve the Bellman equation exactly.
\end{proof}

Proposition \ref{prop:proposition-CT-L2} below examines the constrained problem under a L2 norm
\begin{align}
	 \cT^{\CT,\cD^\pi}[V](s) = \underset{\bpii_s}{\text{maximize}}\qquad & \bw^\T_s \bpii_s &\label{prop:CT-SM-eq3}  \\
	 \text{subject to} \qquad & \sum_{a} \frac{(\pi(a|s) - \overline{\pi}(a|s))^2}{\overline{\pi}(a|s)}  \leq c_s  \nonumber\\
	  & \bpii_s \in \Delta(|\cA|)&\nonumber
\end{align}
\begin{proposition}
\label{prop:proposition-CT-L2}
The optimal value of \eqref{prop:CT-SM-eq3} is equal to 
\[
\min_{\bmu_s\in\bbR^{|\cA|}_+} \left\{\sum_{a} \overline{\pi}(a|s) (w_{sa}+\mu_{sa})+ \sqrt{c_s\sum_a \overline{\pi}(a|s)(w_{sa}+\mu_{sa})^2  }  \right\},
\]
which can be solved \textit{exactly} in $\cO(|\cA|\ln(|\cA|))$.
\end{proposition}
\begin{proof}
We convert the problem into a minimization one and prove the results similarly to Lemma 4.2 in \cite{Iyengar2005robustDP}.
\end{proof}

In summary, we show that, beside the widely-used KL divergence, other types of constraints can be used to control the divergence between policies, noting that Bellman updates under L1 or L2 norms  can be solved exactly, and the complexity can be bounded by $\cO(|\cA|\ln |\cA|)$. In this aspect, L1 and L2 norms are more advantageous to use, as compared  to the KL divergence. However, the use of such norm-based constraints need to be justified by numerical experiments. This is however out-of-scope of the paper and we keep the idea of future work.

\end{document}